\documentclass[11pt]{article}
\usepackage{amsmath, amsthm, amssymb}
\usepackage{amsmath}
\usepackage{amsthm}
\usepackage{amscd,amssymb}
\usepackage[all]{xy}
\usepackage{color}
\usepackage{comment}
\usepackage{appendix}
\usepackage{verbatim}
\usepackage{mathrsfs}
\usepackage{pdfsync}
\date{\today}
\usepackage{graphicx,epsfig}
\newtheorem{thm}{Theorem}[section]

\newtheorem{lem}[thm]{Lemma}
\let\oldproofname=\proofname
\renewcommand{\proofname}{\rm\bf{\oldproofname}}

 \newtheorem{prop}[thm]{Proposition}
  
\theoremstyle{definition}

 \newtheorem{defn}{Definition}[section]
 \newtheorem{claim}{Claim}
  
 \newtheorem{conj}{Conjecture}[section]

\newcommand{\Z}{\mathbb{Z}}
\newcommand{\A}{\mathcal{A}}

\newcommand{\N}{\mathcal{N}}
\newcommand{\M}{\mathcal{M}}

\def\Hom{\text{Hom}}

\begin{document}
\title{Connectedness of Certain Graph Coloring Complexes}
\author{ Nandini Nilakantan\footnote{Department of Mathematics and Statistics, IIT Kanpur, Kanpur-208016, India. nandini@iitk.ac.in.},  Samir Shukla\footnote{{Department of Mathematics and Statistics, IIT Kanpur, Kanpur-208016, India. samirs@iitk.ac.in.}}}
\maketitle

\begin{abstract}
In this article, we consider the bipartite graphs $K_2 \times K_n$. We prove that the connectedness of the complex $\displaystyle \text{Hom}(K_2\times K_{n}, K_m) $ is $m-n-1$ if $m \geq n$ and $m-3$ in the other cases. Therefore, we show that for this class of graphs, $\text{Hom} (G, K_m)$ is exactly $m-d-2$ connected, $m \geq n$, where $d$ is the maximal degree of the graph $G$.
 
\end{abstract}

 

\vspace{.1in}

\hrule

 \section{Introduction}

A classical problem in graph theory is the determination of the chromatic number of a graph 
which finds application in several fields. In 1955, Kneser posed the Kneser conjecture which dealt with the 
computation of the chromatic number of a certain class of graphs \cite{kneser}. This conjecture was proved in 1975 by Lova{\'s}z \cite{l}.
Lova{\'s}z constructed a simplicial complex called the neighborhood complex $\N(G)$ of a graph $G$ to prove the 
Kneser conjecture.

In \cite{BK}, Lova{\'s}z generalised the neighborhood complex by the prodsimplicial complex called the {\it  Hom complex},
 denoted by $\Hom(G,H)$ for graphs $G$ and $H$.  
  In particular, $\Hom(K_2,G)$ (where $K_2$ denotes a complete graph with $2$ vertices)
   and $\N(G)$ are homotopy equivalent. The idea was to be able to estimate the chromatic number
   of an arbitrary graph $G$ by understanding the connectivity of
   the Hom complex from some standard graph into $G$. If $H$ is the complete graph $K_{n}$, then 
    the complexes $\text{Hom} (G, K_n)$ are highly connected. These complexes $\Hom(G, K_n)$ can be considered to be the spaces of $n$- colorings of the graph $G$.  The complexes $\text{Hom} (G,H)$ are not very well understood.
    
Babson and Kozlov made the following conjecture in \cite{BK}.

\begin{conj} \label{conj1}
For a graph $G$ with maximal degree $d$, $\text{Hom}(G,K_n)$ is  at least $n-d-2$ connected.

\end{conj}

 {\v C}uki{\' c} and Kozlov proved this conjecture in \cite{CK}.
They also proved that for those cases where $G$ is an odd cycle, 
$\text{Hom}(G,K_n)$ is $n-4$ connected for all $n \geq 3.$
In \cite{dk1}, it is proved that for even cycles $C_{2m}$,
$\text{Hom}(C_{2m},K_n)$ is $n-4$ connected for all $n \geq 3.$
In \cite{gm}, Malen proved a strong generalisation of \cite{CK}, by showing that $\text{Hom}(G, K_n)$ is at least $m-D-2$ connected, where $D = \max \limits_{H \subset G} \delta(H)$, $\delta(H)$ being the minimal degree of the graph $H$, and the summation is taken over all the  induced subgraphs $H$ of $G$. Since $D \leq d$, this also implies Conjecture \ref{conj1}.

It is natural to ask whether it is possible to classify the class of graphs $G$ for which the Hom complexes $\text{Hom}(G,K_n)$
 are exactly $n- D -2$  connected. In this article, we consider the bipartite graphs $K_2\times K_n$,
 which are $n-1$ regular graphs and $D= d = n-1$.
  Since $\displaystyle \text{Hom}(K_{2}\times K_{n}, K_m) \simeq \text{Hom}(K_{2}, K_m^{K_{n}})$, where $\displaystyle K_{m}^{K_{n}}$ is an exponential graph,
   it is sufficient to determine the connectedness of $\text{Hom}(K_{2}, K_m^{K_{n}})$
   which is the same as the connectedness of the neighborhood complex $\N (K_{m}^{K_n})$.
   The main result of this article is
\begin{thm}\label{main}

 Let $m, n \geq 2$. Then
 \begin{center}
   $\text{conn}(\text{Hom}(K_2 \times K_n, K_m)) = \begin{cases}
             m-n-1 & \text{if \hspace{0.3 cm}$m \geq n$}\\
                             m-3 & \text{otherwise}.\\
                       \end{cases}$
 \end{center}
\end{thm}

In the special case, where $m=n+1$, in \cite{NS},  it is shown that 
\begin{thm} \label{thm1}

 Let $n \geq 4$ and $p$ = $\frac{n! (n-1) n} {2}$. Then   \hspace{5 cm}

   \begin{center}
            $H_k(\mathcal{N}(K_{n+1}^{K_n}); \Z_2) = \begin{cases}
            \ Z_2  & \text{if \hspace{0.3 cm}$k = 0, 1$ or $n-1$}\\
            \ Z_2^{p - n! +1} & \text{ if \hspace{0.3 cm}$k = 2$}\\
                             0 & \text{otherwise}.\\
                       \end{cases}$
  \end{center}
  \end{thm}

\section{Preliminaries}

A  graph $G$ is a  pair $(V(G), E(G))$,  where $V(G)$  is the set of vertices of $G$  and $E(G) \subset V(G)\times V(G)$
the set of edges.
 If $(x, y) \in E(G)$, it is also denoted by $x \sim y$. Here, $x$ is said to be adjacent to $y$.
The {\it degree} of a vertex $v$ is defined by  $\text{deg}(v) =
|\{y \in V(G) \ | \ x \sim y\}|$. Here $| X|$ represents the
cardinality of the set $X$.

 A {\it  bipartite graph} is a graph $G$ with subsets
$X$ and $Y$ of $V(G)$ such that $V(G)=X   \sqcup Y$ and $(v,w) \notin E(G)$ if $\{v,w\} \subseteq X$ or $\{v,w\} \subseteq Y$.
Examples of bipartite graphs include the even cycles $C_{2n}$
where $V(C_{2n})=\{1,2,\dots, 2n\}$ and $E(C_{2n})=\{(i,i+1) \ | \
1\leq i \leq 2n-1\}\cup\{(1,2n)\}$. In this case $X=\{1,3,5,\ldots
,2n-1\}$ and $Y=\{2,4,6,$ $\dots, 2n\}$.
 
A {\it graph homomorphism} from  a graph $G$ to graph $H$ is a function
$\phi: V(G) \to V(H)$ such that, $$(v,w) \in E(G) \implies (\phi(v),\phi(w)) \in E(H).$$

A {\it finite abstract simplicial complex X} is a collection of
finite sets such that if $\tau \in X$ and $\sigma \subset \tau$,
then $\sigma \in X$. The elements  of $X$ are called {\it simplices}
of $X$. If $\sigma \in X$ and $|\sigma |=k+1$, then $\sigma$ is said
to be $k-dimensional$. A $k-1$ dimensional subset of a $k$ simplex $\sigma$
is called a {\it facet} of $\sigma$.
A {\it prodsimplicial complex} is a polyhedral complex each of whose cells is a direct product of simplices (\cite{dk}).

Let $v$ be a vertex of the graph $G$. The {\it
neighborhood  of $v$ } is defined by $N(v)=\{ w \in V(G) \ |  \
(v,w) \in E(G)\}$.  If $A\subset V(G)$, the neighborhood of $A$
is defined as $N(A)= \{x \in  V(G) \ | \ (x,a) \in E(G)
\,\,\forall\,\, a \in A \}$.

The {\it neighborhood complex}, $\N(G)$ of a graph $G$ is the abstract simplicial complex whose vertices are all the non isolated vertices of $G$ and whose simplices are those subsets of $V(G)$ which have a common neighbor.

Consider distinct vertices $u$ and $v$ in $G$ such that $N(u) \subset N(v)$. 
The subgraph $G\setminus \{u\}$ of $G$, where $V(G \setminus \{u\}) = V(G) \setminus \{u\}$ and 
$E(G \setminus \{u\}) = E(G) \setminus \{(v, w) \in E(G) \ | \ \text{at least one of}\, v  \, \text{or} \, w \, 
\text{is} \, u\}$, is called a {\it fold} of $G$.

For any two graphs $G$ and $H$, $\Hom(G,H)$ is the polyhedral complex whose cells are indexed by all functions $\eta: V(G) \to 2^{V(H)}\setminus \{\emptyset\}$, such that if $(v,w) \in V(G)$, then $\eta(v) \times \eta(w) \subset E(H)$.

 Elements of  $\Hom(G,H)$ are called cells and are denoted by $(\eta(v_1), \dots, $ $ \eta(v_k))$,
 where $V(G)=\{v_1, \dots, v_k\}$.
A cell $(A_1, \dots, A_{k})$ is called a {\it face} of $B=(B_1,
\dots B_k)$, if $A_{i} \subset B_{i}$ $\forall\, 1\leq i\leq k$. The
Hom complex is often referred to as a topological space. Here, we
are referring to the geometric realisation of the order complex of
the poset. The simplicial complex whose facets are the chains of the
Poset $P$ is called the order complex of $P$.

A topological space $X$ is said to be $n$-{\it connected} if $\pi_\ast(X)= 0$ for all $ \ast \leq n$.
 By convention, $\pi_0(X)= 0 $ means $X$ is connected.
The connectivity of a topological space $X$ is denoted by
$\text{conn}(X)$, {\it i.e.}, $\text{conn}(X)$ is the largest
integer $m$ such that $X$ is $m$-connected.
If $X$ is a non empty, disconnected space, it is said to be $-1$ connected and if it is empty,
it is said to be $-\infty$  connected.

We now review some of the constructions related to the existence of an {\it internal hom}
which is related to the categorical product.  All these details can be found in \cite{pj, cj, sm}.

\begin{itemize}
\item
The {\it categorical product} of two graphs $G$ and $H$, denoted by $G\times H$ is the graph where
$V(G\times H)=V(G)\times V(H)$ and  $(g,h) \sim (g',h')$ in $G\times H$  if $g \sim g'$  and $h \sim h'$ in $G$
and $H$ respectively.

\item
If $G$ and $H$ are two graphs, then the {\it exponential graph} $\displaystyle H^{G}$ 
is defined to be the graph where $\displaystyle V(H^{G})$ contains all the set maps from $V(G)$ to $V(H)$.
Any two vertices  $f$  and $f'$ in $\displaystyle V(H^{G})$ are  said to be adjacent, if $ v \sim v'$  in $G$ implies that $f(v) \sim f'(v')$ in $H$.

\end{itemize}

Using tools from poset topology (\cite{bj}), it can be shown that given a poset $P$ and a poset map
$c:P \rightarrow P$ such that $ c\circ c =c$ and $c (x) \geq x \, \forall \,x \in P$, there is a strong deformation retract induced by $c :P \rightarrow c(P)$ on the relevant spaces. Here, $c$ is called the {\it closure map}.

From \cite[Proposition 3.5]{ad} we have a relationship between the exponential graph and the categorical product in the $\text{Hom}$-complex.

\begin{prop}
Let $G$, $H$ and $K$ be graphs. Then  $\displaystyle \text{Hom}(G \times H, K)$ can be included in $\displaystyle \text{Hom}(G, K^{H})$  so that $\displaystyle \text{Hom}(G \times H, K)$ is the image of the closure map on $\displaystyle \text{Hom}(G, K^{H})$. In particular, there is a strong deformation retract  $\displaystyle |\text{Hom}(G \times H, K)| \hookrightarrow |\text{Hom}(G, K^{H})|$.
\end{prop}

From \cite[Proposition 5.1]{BK} we have the following result which allows us to replace a graph by its fold in the $\text{Hom}$ complex.

\begin{prop}\label{fold}
Let $G$ and $H$ be graphs such that $u, v $ are distinct vertices of $G$ and $N(u) \subset N(v)$. The inclusion $i :G \setminus \{u\}  \hookrightarrow G$ respectively,
 the homomorphism  $\phi:G \rightarrow G \setminus \{u\} $ which maps $u$ to $v$ and fixes all the other vertices,
  induces the homotopy equivalence  $i_{H} :\text{Hom}(G,  H) \rightarrow \text{Hom}(G \setminus \{u\},  H)$,‪ 
  respectively  $\phi_{H} :\text{Hom}(G \setminus \{u\},  H) \rightarrow \text{Hom}(G,  H)$.
\end{prop}

\section{Tools from Discrete Morse Theory}

We introduce some tools from Discrete Morse Theory  which have been
used in this article. R. Forman in \cite{f} introduced what has now
become a standard tool in Topological Combinatorics, Discrete Morse
Theory. The principal idea of Discrete Morse Theory (simplicial) is
to pair simplices in a complex in such a way that they can be
cancelled by elementary collapses. This will reduce the original
complex to  a homotopy equivalent complex, which is not necessarily
simplicial, but which has fewer cells. More details of discrete
Morse theory can be found in \cite{jj} and \cite{dk}.

\begin{defn}
A {\it partial matching} in a Poset $P$ is a subset $\M$ of $P \times P$ such that
\begin{itemize}
\item $(a,b) \in \M$ implies $b\succ a$, {\it i.e. $a<b$ and $\not\exists \,c$ such that $a<c<b$}.
\item Each element  in $P$ belongs to at most one element in $\M$.
\end{itemize}
\end{defn}

In other words, if $\M$ is a {\it partial matching} on a Poset $P$
then, there exists  $A \subset P$ and an injective map $f: A
\rightarrow P\setminus A$ such that $f(x)\succ x$ for all $x \in A$.

\begin{defn}
An {\it acyclic matching} is a partial matching  $\M$ on the Poset $P$ such that there does not exist a cycle
\begin{eqnarray*}
f(x_1)  \succ x_1 \prec f( x_2) \succ x_2  \prec f( x_3) \succ x_3 \dots   f(x_t) \succ x_t  \prec f(x_1), t\geq 2.
\end{eqnarray*}

\end{defn}

If we have an acyclic partial matching on $P$, those elements of $P$ which do not belong to the matching are said to be {\it critical }. To obtain the  desired homotopy equivalence, the following result is used.

\begin{thm} (Main theorem of Discrete Morse Theory)\label{dmt}\cite{f}
\mbox{}

Let $X$ be a simplicial complex and let $\A$ be an acyclic matching such that the empty set is not critical. 
Then, $X$ is homotopy equivalent to a cell complex which has a $d$ -dimensional cell for each $d$ -dimensional critical face of $X$ together with an additional $0$-cell.

\end{thm}

\section{Main Result}

To prove  Theorem \ref{main}, we first construct
an acyclic matching on the face poset of $\N (K_{m}^{K_{n}})$.
We then consider the corresponding Morse Complex to determine the connectedness of 
$\Hom(K_2 \times K_n, K_m)$.
In this article $n \geq 3, m-n = p \geq 1$ and $[k]$ denotes the
set $\{1,2,\dots, k\}$. Any vertex in the exponential graph
$K_{m}^{K_n}$ is a set map $f: [n] \rightarrow [m]$.

\begin{lem} \label{lemma2}

The graph $K_{m}^{K_n}$ can be folded onto the graph $G$, where the vertices
$f \in V(G)$ have images of cardinality either 1 or $n$.

\end{lem}
\begin{proof} Consider the vertex $f$ such that  $1<|$Im $f|$ $< n$.  Since $f$ is not injective
there exist distinct $i, j \in [n]$ such that $f(i)=f(j) =
\alpha$. Consider  $\tilde{f} \in V(K_{m}^{K_n})$ such that
$\tilde{f}([n])=\alpha$. By the definition of the exponential graph, any
neighbor $h$  of $f$ will not have $\alpha$ in its image  and
therefore $h$ will be a neighbor of $\tilde{f}$ thereby showing that $N(f) \subseteq N(\tilde{f})$.
 $K_{m}^{K_n}$ can be folded to the subgraph $K_{m}^{K_n}\setminus \{f\}$ of $K_{m}^{K_n}$.
  Repeating the above argument for all noninjective, non constant maps from $[n]$ to $[m]$, $K_{m}^{K_n}$ 
can be folded to the graph $G$ whose vertices are either constant maps or injective maps from $[n]$ to $[m]$.
\end{proof}

From Proposition \ref{fold}, we observe that $\N (K_{m}^{K_n})  \simeq \N (G)$. Hence, it
is sufficient to study the homotopy type of $\N (G)$.\\
We now fix the following notations. \\
If $f \in V(G)$ and $f([n])=\{x\}$, then $f$ is denoted by $<x>$. 
In the other cases the string $a_1a_2\dots a_n$ denotes the vertex $f$ where
$a_i =f(i)$,  $1\leq i \leq n.$ Hence, if the notation $a_1a_2\dots a_{n}$ is used, 
it is understood that  for $1\leq i<j\leq n$, $a_i \neq a_j$.
Let $\sigma = \{f_1, \ldots, f_q\} \subset V(G)$. Then 
$X_i^{\sigma} = \{f_j(i) \ | \ 1 \leq j \leq q\}$,
$A_i^{\sigma} = [m] \setminus \bigcup \limits_{j \neq i}X_{j}^{\sigma}$ and 
$X_{\sigma} = \bigcup\limits_{i = 1}^{n} X_{i}^{\sigma}$.

\subsection{Construction of an acyclic matching}
Let $(P, \subset)$ be the face poset of $\N(G)$, $S_1= \{\sigma \in P \ |  <1> \ \notin \sigma , \sigma \, \cup <1> \ \in P \}$
and the map $\mu_1 : S_1 \rightarrow P \setminus S_1$ be defined by $\mu_1(\sigma) = \sigma \, \cup <1>$.\\
Let $S_1'= P \setminus \{S_1, \mu_1(S_1)\}$ and for  $2 \leq i \leq m ,$ define 
$S_i = \{\sigma \in S_{i-1}' \ | <i> \ \notin \sigma , \sigma \, \cup <i> \ \in S_{i-1}'\}$ and $\mu_i : S_i 
\rightarrow S_{i-1}' \setminus S_i$ by  $\mu_i(\sigma) = \sigma \, \cup <i>,$ where $S_i'= 
S_{i-1}' \setminus \{S_i, \mu_i(S_i)\}.$ Since $S_i \cap S_j  = \emptyset \ \forall \ 1 \leq i \neq j \leq m$, if 
$\sigma \in S = \bigcup \limits_{i= 1}^{m} S_i$, then $\sigma $ belongs to exactly one $S_i.$
Define $\mu : S  \rightarrow P \setminus S$ by $\mu(\sigma) = \mu_i(\sigma)$ where $\sigma \in S_i$.
By the construction, $\mu$ is an injective map and  hence $\mu$ is a well defined partial matching on the poset $P$.
 
\begin{prop}
 $\mu$ is an acyclic matching on $P$.
  \end{prop}

\begin{proof}

Assume that $\mu$ is not an acyclic matching, {\it i.e.}, there exist distinct simplices $\sigma_1, \ldots,$ $  \sigma_t$ in $ S$ 
such that $\mu(\sigma_i) \succ  \sigma_{i+1  (\text{mod} \ t)}, 1 \leq i \leq t$.

\begin{claim}\label{cl1} $\sigma_{i+1  (\text{mod} \ t)} = \mu(\sigma_i) \setminus \{f\}$, where $|\text{Im}\,f|=1$.

\end{claim}

$\mu(\sigma_i) \succ  \sigma_{i+1  (\text{mod} \ t)}$ implies that 
 $\sigma_{i+1} = \mu(\sigma_i) \setminus \{f\}$, for some $f\in V(G)$. If $|\text{Im}\,f|=n$, then $f \notin \sigma_{i+1 (\text{mod} \ t)} \Rightarrow 
 f \notin \mu(\sigma_{i+1 (\text{mod} \ t)})\Rightarrow f \notin \sigma_{i+2 (\text{mod} \ t)}$. Hence, $f \notin \mu(\sigma_i)$, a contradiction. So $f = \ <j>$ for some $j \in \{1, \ldots, m\}$. 
 
 Let $k \in \{1, \ldots, m\}$ be the least integer such that $\{\sigma_1, \ldots, \sigma_t\} \cap S_k \neq \emptyset$.
 Without loss of generality assume that $\sigma_1 \in S_k$, {\it i.e.} $<k> \ \notin \sigma_1$ and $\mu(\sigma_1) 
 = \sigma_1 \cup <k>$. Since $\sigma_{1} \neq \sigma_{2}$, from Claim \ref{cl1}, 
 $\sigma_2 = \mu(\sigma_1) \setminus \{<i>\}$,  $i \neq k$. Therefore, $<k>  \in \sigma_2$.
 If $<k> \ \in \sigma_i, 2 \leq i \leq t,$ then $<k> \ \in \mu(\sigma_t)$. 
$\mu(\sigma_1) \setminus \{<k>\}  = \sigma_1 = \mu(\sigma_t) \setminus \{<k>\} \Rightarrow \mu(\sigma_1) = \mu(\sigma_t)$, 
 a contradiction  (since $\sigma_1 \neq \sigma_t$ and $\mu$ is injective).

 Let $l \in [m]$ be such that $<k> \ \notin \sigma_l$ and  $<k> \ \in \sigma_j, 2\leq  j <l$.   Here,  $<k> \ \in
 \mu(\sigma_{l-1})$ and therefore, $\sigma_l = \mu(\sigma_{l-1}) \setminus \{<k>\}$. 
Since $\sigma_l,\mu( \sigma_{l-1}) \notin S_j, \mu_{j}(S_j)$ $\forall\, j <k$, we see that  $\sigma_l, \mu(\sigma_{l-1}) \in S_{k-1}'$, 
which implies that $\sigma_{l} \in S_{k}$ (since $\mu(\sigma_{l-1})= \sigma_l \cup <k>$).
Here, $\sigma_{l} \neq \sigma_{l-1}$ and 
 $\mu(\sigma_l) = \mu(\sigma_{l-1})$, a contradiction. 
 Thus our assumption that $\mu(\sigma_{i}) \succ \sigma_{i+1 (\text{mod} \ t)}, 1 \leq i \leq t$ is incorrect {\it i.e.}
 there exists no such cycle in $P$. Hence, $\mu$ is an acyclic matching.
 
 \end{proof}

 \subsection{Critical Cells}
 A subset $\sigma$ of $V(G)$ is a simplex in $\N(G)$ if and only if 
 $A_{i}^{\sigma} \neq \emptyset \ \forall \ i \in [n]$.
 In this section, we characterize the critical cells of $\N(G)$. If there exists $z_1 \ldots z_n \in N(\sigma)$ such that
 $1 \notin \{z_1, \ldots, z_n\}$, then $\sigma \, \cup <1> \ \in \N(G)$, which implies that 
 $\sigma$ is not a critical cell. Further,
 if $\sigma \in N(<x>)$ for some $x \neq 1$, then $<1> \ \in N(\sigma)$ and $\sigma  \in S_1$. Therefore, for $\sigma$ to be a critical 
 cell, either $N(\sigma) = \ <1>$ or $\sigma \subset N(z_1 \ldots z_n)$ with $1\in \{z_1,z_2,\dots z_{n}\}$.
 
 \noindent{\bf{Case 1.}}  $\sigma \subset N(z_1 \ldots z_n)$, where $z_{k}=1$ for some $k \in \{1,2,\dots,n \}$.
 
 In this case, clearly, $\sigma = \{<x_1>, \ldots, <x_l>\}\ \cup \tau$, where $\tau = \{f_1, \ldots, f_q\}$ with 
 $|\text{Im}\, f_i| = n $ $\forall\ i \in \{1, \dots, q\}$ and $x_1, \dots x_l \in [m] \setminus \{z_1, \ldots, z_n\}$.
 Since $\sigma \sim z_1 \ldots z_n$ and $<z_i>$ is not a neighbor of $z_1 \ldots z_n$ for all
 $i \in [n]$ we see that $<z_i> \ \notin \sigma$. If $z_r \in X_{\sigma}$ for some $r \in [n]$, 
 then 
$z_r \in X_{\tau}$, which implies that there exists $f \in \tau$ such that $z_r \in \text{Im} \,f$. Since
$f \sim z_1 \ldots z_n$, $f(i) \neq z_r \ \forall \ i \neq r$, which implies that  $f(r) = z_r$. 

\begin{prop} \label{prop1}
For  $x\in [m]$ with $x \neq z_{i}$, $\forall ~i \in [n]$ and $<x> \ \notin \sigma$, define
\[
  a=\left\{\def\arraystretch{1.2}%
  \begin{array}{@{}c@{\quad}l@{}}
    \text{min} \{x,z_r\}& \text{if} ~~x \in X_r^{\tau} \setminus \bigcup \limits_{r \neq j} X_{j}^{\tau} \\
    x & \text{if} ~~x \in X_i^{\tau} \cap X_{j}^{\tau}, i \neq j.\\
  \end{array}\right.
\]
 Let $\eta$  be the set $\{\sigma \, \cup <a_1>, \ldots, <a_s> \} \setminus \{ <b_1>, \ldots, <b_t>\}, s, t \geq 0, a_i, b_j
 < a, 1 \leq i \leq s, 1 \leq j \leq t$. Then
  \begin{itemize}
   \item [(i)] $\eta \in \N(G) \Longleftrightarrow \eta \, \cup <a> \ \in \N(G).$
   \item[(ii)] $\eta \in S_1 \Longleftrightarrow \eta \, \cup <a> \ \in S_1.$
  \end{itemize}
\end{prop}
\begin{proof}
 If $\eta \, \cup <a> \ \in \N(G) \ \text{or} \ S_1$, then $\eta \in \N(G) \ \text{or} \ S_1$, respectively.\\
 \textbf{Case (i)} $a = \text{min}\{x, z_r\}$.\\
 Since $ a \leq x, z_r \in X_r^{\tau} \setminus \bigcup \limits_{r \neq j} X_{j}^{\tau}$ and 
 $<x>, <z_r> \ \notin \sigma$, we observe that $x, z_r \in X_r^{\eta} \setminus \bigcup \limits_{r \neq j} X_{j}^{\eta}$.
 Any neighbor $f$ of $\eta$ has the property that $f(i) \neq x, z_r$ ({\it i.e.} $f(i) \neq a$)
 if $i \neq r$. If $<y> \ \sim \eta,$ {\it i.e.} 
 $\eta \in \N(G)$, then $y \neq a$ and  $<y> \ \sim \ <a>$ implying that $<y> \ \sim \eta \, \cup <a>$ and thus 
 $\eta \, \cup <a> \ \in \N(G)$. \\
 If $\eta \in S_1$ and $<y> \ \sim \eta \, \cup <1> $, then  $<y> \ \sim \ <a>$ and $<1>$. Therefore 
 $\eta \, \cup <a> \cup <1> \ \in \N(G)$ and  $\eta \, \cup <a> \ \in S_1.$\\
 Let $x_r = \text{max}\{x, z_r\}$. Clearly, $a, x_r \in X_r^{\tau} \setminus \bigcup \limits_{i \neq r} X_i^{\tau}$.
 For any neighbor $y = y_1 \ldots y_n$ of $\eta$, $y_i \neq x_r, a \ \forall \ i \neq r$ and so 
 $y_1 \ldots y_{r-1} x_r y_{r+1} \ldots y_n \sim \eta$ and $<a>$. Therefore $\eta \, \cup <a> \ \in \N(G)$ if $\eta \in \N(G)$.\\
 If $\eta \in S_1$ and $y \sim \eta$, then $y_i \neq 1 \ \forall \ i \in [n]$.
 Since $x_r \neq 1, y_1 \ldots y_{r-1} x_r y_{r+1}$ $ \ldots y_n \sim \eta, <a>$ and $<1>$. Thus, $\eta \, \cup <a> \ \in S_1$. 
 
 \noindent{\bf{Case (ii)}} $x = a \in X_i^{\tau} \cap X_{j}^{\tau}, i \neq j$.\\
 In this case, $a \neq z_t$ and $a \notin A_t^{\eta} \ \forall \ t \in [n]$. Since $A_{t}^{\eta \cup <a>} =
 A_t^{\eta} \setminus \{a\}$
and $a \notin A_t^{\eta}$, we get $A_t^{\eta \cup <a>} = A_t^{\eta}$. If $f \in N(\eta)$, then 
$ \forall \ t \in [n]$, $f(t) \in A_t^{\eta} = A_t^{\eta \cup <a>}$ implying that $f \in N(\eta \, \cup <a>)$. 
$N(\eta \, \cup <a>) \subset N(\eta)$ always, and therefore, $N(\eta) = N(\eta \, \cup <a>)$. $(i)$ and $(ii)$ now follow.
\end{proof}

 \begin{prop} \label{propmain}
 Let $a$ be as defined in Proposition \ref{prop1},  $\sigma \notin S_t \cup \mu_t(S_t) $ $  \forall \ t < a$ and 
 $\sigma \, \cup <a> \ \in S_{i_1} \cup \mu_{i_1}(S_{i_1})$,  $i_1 < a$. 
Then,  there exists $s \in \{2, \dots, m\}$ and a cell 
 \begin{center}
   $\xi_s = \sigma \cup\begin{cases}  \{<a>, <b_1>, \ldots, <b_r>\}
 \setminus \{<l_1>, \ldots, <l_t>\} & \text{if} ~ \text{s}~ \text{is} ~  \text{even} \\
               \{<b_1>, \ldots, <b_r>\}
 \setminus \{<l_1>, \ldots, <l_t>\}& \text{if} ~ \text{s} ~ \text{is} ~  \text{odd},
                       \end{cases}$
 \end{center}
 
 where $r, t \geq 0$ and $b_1, \ldots, b_r, l_1, \ldots, l_t < a$ such that  $\xi_{s} \in S_{1}\cup \mu_{1}(S_1)$ but
 $\xi_s \setminus \{<a>\}$ (if  $<a> \ \in \xi_s$) or $ \xi_s \, \cup <a>$ ( if $<a> \ \notin \xi_s$)
 $\notin S_1 \cup \mu_1(S_1)$.
 \end{prop}
\begin{proof} 
Clearly, $i_1 > 1$ as  $\sigma \, \cup <a> \ \in S_{1} \cup \mu_1(S_1)$ implies that $\sigma \in S_1 \cup \mu_1(S_1)$, a contradiction.
Define the cell
\begin{center}
   $\xi_1 = \begin{cases}
             \sigma \setminus \{<i_1>\} & \text{if} ~~ <i_1> \ \in \sigma\\
              \sigma \, \cup <i_1> & \text{if} ~~ <i_1> \ \notin \sigma.\\
                       \end{cases}$
 \end{center}
Since,  $\sigma \ \cup <a> \ \in S_{i_1} \cup \mu_{i_1}(S_{i_1})$, $\xi_{1} \ \cup <a>$ is always a simplex, thereby showing that $\xi_{1} \in \N (G)$.
 If  $\xi_1 \in S_{i_1} \cup \mu_{i_1}(S_{i_1})$, {\it i.e.} $\xi_{1} \notin S_{j} \cup \mu_{j}(S_{j})$ $\forall \ j <i_1$, then $\sigma$ will belong to $S_{i_1} \, \cup \mu_{i_1}(S_{i_1})$, which contradicts the hypothesis. Therefore, $\exists \ i_2 < i_1$ such that 
 $\xi_1 \in S_{i_2} \cup \mu_{i_2}(S_{i_2})$. By Proposition \ref{prop1}(i), $\xi_1 \, \cup <a>$, which is $\{\sigma \, \cup <a>\} \setminus \{<i_1>\}$
 or $\sigma \, \cup <a> \cup <i_1>$, belongs to $S_{i_1} \cup \mu_{i_1}(S_{i_1}).$ So, if $i_2 =1$ the result holds. 
 Let $i_2 > 1$.
 Define the cell 
 \begin{center}
   $\xi_2 = \begin{cases}
              \{\xi_1 \, \cup <a>\} \setminus \{<i_2>\} & \text{if} ~~ <i_2> \ \in \sigma\\
              \xi_1 \, \cup <a> \cup <i_2> & \text{if} ~~ <i_2> \ \notin \sigma.
                       \end{cases}$
 \end{center}

 By  Proposition \ref{prop1}, $\xi_2 \in \N(G)$. Since 
 $\xi_1 \, \cup <a> \ \notin S_{i_2} \cup \mu_{i_2}(S_{i_2})$, there exists $i_3 < i_2$ such that
 $\xi_2 \in S_{i_3}\cup \mu_{i_3}(S_{i_3})$. Here, $\xi_2 \setminus \{<a>\}$ which is $\xi_1 \setminus \{<i_2>\}$ or
 $\xi_1 \, \cup <i_2> \ \in S_{i_2} \cup \mu_{i_2}(S_{i_2})$, which implies that the result holds if $i_3 = 1$.
 Inductively, assume that there exists $l, 1  <l < m$, where $ i_{l+1} < i_{l} < \ldots < i_1 < a$ and 
 $\xi_{t} \in S_{i_{t+1}}\cup \mu_{i_{t+1}}(S_{i_{t+1}})$, $1 \leq t \leq l$  such that
\begin{center}
   $\xi_{t} = \begin{cases}
              \{\xi_{t-1} \, \cup <a>\} \setminus \{<i_{t}>\} & \text{if} ~ <i_{t}> \ \in \xi_{t-1},~ t~\text{is}~\text{even}   \\
              \xi_{t-1} \, \cup <a> \cup <i_{t}> & \text{if} ~~ <i_{t}> \ \notin \xi_{t-1},~ t~\text{is}~\text{even}\\
                                     \xi_{t-1} \setminus \{<a>,<i_{t}>\} & \text{if} ~~ <i_{t}> \ \in \xi_{t-1},~ t~\text{is}~\text{odd}\\
              \{\xi_{t-1} \, \cup <i_{t}>\} \setminus \{<a>\} & \text{if} ~~ <i_{t}> \ \notin \xi_{t-1},~ t~\text{is}~\text{odd}.
                       \end{cases}$
 \end{center}

Since $\xi_{l} \setminus \{<a>\} ( \text{or} ~~ \xi_{l} \cup <a>) =  \xi_{l-1} \setminus \{<i_l>\}$ or 
 $\xi_{l-1} \, \cup <i_l> \ \in S_{i_l} \cup \mu_{i_l}(S_{i_l})$ and $i_l > i_{l+1}$, the result holds if $i_{l+1}=1$. If $i_{l+1} >1$, by induction
 the result follows.
 
 \end{proof}

We first get some necessary conditions for $\sigma$ to be  a critical cell.
\begin{lem} \label{theorem1} Let $\sigma$ be a critical cell. Then
 \begin{itemize}
  \item [(i)] $X_{\sigma} = [m] \ \text{or} \ [m] \setminus \{1\}.$
  \item [(ii)] $x \in [m] \setminus \{z_1, \ldots, z_n\} \Rightarrow \ <x> \ \in \sigma$.
 \end{itemize}
 \end{lem}
 \begin{proof}
 A critical cell $\sigma$ does not belong to $S_i \cup \mu_i(S_i)$ for all $i \in \{1, \ldots, m\}.$
 \begin{itemize}
  \item[(i)] If $x \in [m] \setminus \{1\}$ such that $x \notin X_{\sigma}$,
  then $\sigma, <1> \ \sim \ <x>$,
  thereby implying that $\sigma \in S_1$. Hence $[m] \setminus \{1\} \subset X_{\sigma}$.
  \item[(ii)] Let $x \in [m] \setminus \{z_1, \ldots,z_n\}$ such that
   $<x> \ \notin \sigma$. Since $z_1 \ldots z_n$ is a neighbor of both $<x> $ and $\sigma$, we see that 
   $\sigma \, \cup <x> \ \in \N(G)$. $x \neq 1 \in X_{\sigma}, <x> \ \notin \sigma$ implies that there exists
    $i \in [n]$ such that $x \in X_i^{\tau}$.
   If $x \in X_{k}^{\tau} \setminus \bigcup \limits_{j \neq k} X_{j}^{\tau}$, then 
   $\sigma \sim z_1 \ldots z_{k-1} x z_{k+1} \ldots z_n$, which implies that 
   $\sigma \in S_1$. Thus $x \notin X_k^{\tau} \setminus \bigcup\limits_{j \neq k} X_j^{\tau}$.
      If $a$ is as defined in  Proposition \ref{prop1}, then  $\sigma \, \cup <a> \ \in \N(G)$ and  $\xi_s \in S_1 \Rightarrow \xi_s \, \cup <a>$ or $\xi_s \setminus \{<a>\}  \in S_1$. This 
 contradicts   Proposition \ref{propmain} and thus, 
 $<x> \ \in \sigma$ if $x \in [m] \setminus \{z_1, \ldots, z_n\}$.

\end{itemize}
\end{proof}

 \begin{lem} \label{lemma5}
  Let $\sigma$ be a critical cell. Then
 \begin{itemize}
  \item[(i)] $X_r^{\tau} \cap X_{s}^{\tau} = \emptyset$ if $ r \neq s$.
 \item[(ii)] $x \in X_i^{\tau} \Rightarrow x \geq z_i, i \in \{1, \ldots, n\}$.
\item[(iii)]  $\exists \ a \in X_k^{\tau}$ such that $ a <  \text{min} \{z_1, \ldots, \widehat{z_k}, \ldots,  z_n\}$.
 \end{itemize}
 \end{lem}

 \begin{proof} \hspace{5cm}

  \begin{itemize}
   \item[(i)] Assume that $x \in X_r^{\tau} \cap X_s^{\tau}.$ 
   For any $f \in \tau, f(i) \neq z_j \ \forall \ i \neq j,  x \in [m] \setminus \{z_1, \ldots, z_n\}$ and $x \neq 1$.
   Therefore, $<x> \ \in \sigma$ by Lemma \ref{theorem1}.
   Since $<x> \ \notin \sigma \setminus \{<x>\}$, Proposition \ref{prop1} holds, when $\sigma$ is replaced by
   $\sigma \setminus \{<x>\}$ in the definition of $\eta$.
   
   Since $\sigma \notin S_i \cup \mu_i(S_i)$ $\forall \ i $, we see that $\sigma \setminus \{<x>\} \notin S_x \cup \mu_x(S_x)$. Therefore there exists $i < x$ such that $\sigma \setminus \{<x>\} \in S_{i} \cup \mu_{i}(S_{i})$.   
   Considering $\sigma \setminus \{<x>\}$ instead of $\sigma \ \cup <a>$ in  Proposition \ref{propmain} and by an argument similar to that in the proof of Proposition \ref{propmain}, we see that
  there exists $s\in\{1, \dots  , m\}$ and a cell $\xi_s$ such that
 \begin{center}
 $\xi_{s} = 
   \sigma  \cup \begin{cases}  \{<b_1>, \ldots, <b_r>\}
 \setminus \{<x>, <l_1>, \ldots, <l_t>\} & \text{if} ~ s~ \text{is} ~  \text{even} \\
               \{<b_1>, \ldots, <b_r>\}
 \setminus \{ <l_1>, \ldots, <l_t>\}& \text{if} ~ s ~ \text{is} ~  \text{odd},
                       \end{cases}$
 \end{center}
 
 where $r, t \geq 0$ and $b_1, \ldots, b_r, l_1, \ldots, l_t < x$ such that,  $\xi_{s} \in S_{1}\cup \mu_{1}(S_1)$ but
 $\xi_s \setminus \{<x>\}$ (if  $<x> \ \in \xi_s$) or $ \xi_s \, \cup <x>$ (if $<x> \ \notin \xi_s$)
 $\notin S_1 \cup \mu_1(S_1)$. But, this  is not possible by  Proposition \ref{prop1}.
  Hence $X_r^{\tau} \cap X_s^{\tau} = \emptyset$.

  \item[(ii)] Let $x \in X_i^{\tau}$ such that $x < z_i$. As in the case $(i)$, $x \in [m] \setminus \{z_1, \ldots, z_n\}$
  and hence $<x> \ \in \sigma$. By exactly the same proof as the one above, we get a contradiction and thus $x \geq z_i$.

  \item[(iii)] Let $z = \text{min} \{z_1, \ldots, \widehat{z_k}, \ldots, z_n\}$ and
  $a = \text{min}\{x \ | \ x \in X_{k}^{\tau}\}$. If $X_{\sigma} = [m],$ then $a = 1$ and the result follows. 
  Let $X_{\sigma} = [m] \setminus \{1\}$ and  $a \geq z$. $\sigma, <z> \ \sim \ <1>$ implies that $\sigma \, \cup <z> \ \in \N(G)$
  and $\sigma$ is a critical cell implies that $\sigma \, \cup <z> \ \notin S_1$. If $z = 2$, then 
  $\sigma, \sigma \, \cup <2> \ \notin S_1$ and hence $\sigma \in S_2$, which is not possible. Therefore, $z > 2$.
  Choose $t$, $1 < t < z$. Here, $t \neq z_i \ \forall \ i$ and therefore $t \in [m] \setminus \{z_1, \ldots, z_n\}$
  and $t \notin X_i^{\tau} \ \forall \ i \in [n]$ (by $(ii)$). Since $t \notin X_{\sigma \setminus <t>}$ and $t \neq 1$,
  $\{\sigma \, \cup <z>\} \setminus \{<t>\} \in S_1$, which implies $\sigma \, \cup <z> \ \notin \mu_t(S_t) \ \forall \ t < z$.
  Therefore, $\sigma \, \cup <z> \ \in \mu_{z}(S_{z})$, which implies that $\sigma$ is not a critical cell. Thus, $a < z$.
  \end{itemize}
 \end{proof}

 We now consider those cells $\sigma$ for which $X_{\sigma} = [m]$ or $[m] \setminus \{1\}$ and 
 $<x> \ \in \sigma$ if $ x \in [m] \setminus \{z_1, \ldots, z_n\}$. We now get sufficient conditions for $\sigma$ to be 
 a critical cell.
 \begin{lem} \label{lemma6} $\sigma$ is a critical cell if it satisfies the following conditions 
 \begin{itemize}
     \item[(i)] $X_r^{\tau} \cap X_{s}^{\tau} = \emptyset$ if $ r \neq s$.
 \item[(ii)] $x \in X_i^{\tau}$ implies $x \geq z_i, i \in \{1, \ldots, n\}$.
\item[(iii)]  $\exists \ a \in X_k^{\tau}$, where $ a <  \text{min} \{z_1, \ldots, \widehat{z_k}, \ldots,  z_n\}$. 
 \end{itemize} 
 \end{lem}

\begin{proof} If $x \in [m] \setminus \{z_1, \ldots, z_n\}$ such that $x \notin X_{\sigma \setminus <x>}$,
then $x \neq 1$ and $\sigma \setminus \{<x>\} \sim \ <x>$. Thus, $\sigma \setminus \{<x>\} \in S_1$. 
In this case, $\sigma \notin S_x \cup \mu_x(S_x)$. We consider the following two cases.

\noindent{\bf{Case (i) :}} $y \in \{z_1, \ldots, z_n\}$.

\noindent{\bf{(a)}} $X_{\sigma} = [m]$.

In this case, $A_i^{\sigma} = \{z_i\} \ \forall \ i \in \{1, \ldots, n\}$. Therefore, for each 
$i \in [n], A_i^{\sigma \cup <z_i>}$ $ = \emptyset$. Thus, $\sigma \, \cup <z_i> \ \notin \N(G)$ which implies that
$\sigma \notin S_{z_i} \cup \mu_{z_i}(S_{z_i})$.

\noindent{\bf{(b)}} $X_{\sigma} = [m] \setminus \{1\}$.

Here, $A_i^{\sigma} = \{1,z_i\} \ \forall \ i \neq k$ and $A_k^{\sigma} = \{1\}$. Since $A_k^{\sigma \cup <1>} = 
\emptyset$, $\sigma \notin S_1$.
Let $a = \text{min} \{x \ | \ x \in X_k^{\tau}\}$. Since $X_{\sigma} = [m] \setminus \{1\}, a \neq 1$ and by 
$(iii), a < z_i \ \forall \ i \neq k$. Hence, $<a> \ \in \sigma$.

If $l \neq k$ then, $\sigma \, \cup <z_l> \ \sim \ <1>$ and so $\sigma \, \cup <z_l> \ \in \N(G)$.
$A_l^{\sigma \cup <z_l> \setminus <a>} = A_l^{\sigma \cup <z_l>} = \{1\}$ and therefore both
$\{\sigma \, \cup <z_l> \} \setminus \{<a>\}$ and $\sigma \, \cup <z_l> \ \notin S_1$. If $a = 2$, then  
$\sigma \, \cup <z_l> \ \in S_2 \cup \mu_2(S_2)$ (since $\sigma \cup <z_l>, 
\{\sigma \, \cup <z_l>\} \setminus \{<2>\} \notin S_1$), and 
therefore $\sigma \ \notin S_{z_l} \cup \mu_{z_l}(S_{z_l})$. 

Assume $a > 2$ and choose $t$, such that $1 < t < a$. Clearly $t \neq z_i \ \forall \ i \neq k$ and 
$<t> \ \in \sigma$. Since $t < z_i \ \forall \ i \neq k$, from $(ii) \ t \notin X_i^{\tau} \ \forall \ i \neq k$.
Since $a$ is the least element in $X_k^{\tau}$ and $t < a, t \notin X_k^{\tau}$ and thus $t \notin X_{\tau}$.
Therefore, $\{\sigma \, \cup <z_l> \} \setminus \{<t>\}, \{\sigma \, \cup <z_l>\} \setminus \{<a>, <t>\}$ will both be 
neighbor of $<t>$ and hence, they both belong to $S_1$.
Neither $\sigma \, \cup <z_l>$ nor $\{\sigma \, \cup <z_l>\} \setminus
\{<a>\}$ belong to $S_t \cup \mu_t(S_t) \ \forall \ t < a$. Therefore
 $\sigma \, \cup <z_l> \ \in S_a \cup \mu_a(S_a)$. We can now conclude that $\sigma \notin S_{z_i} \cup \mu_{z_i}(S_{z_i}) \
 \forall \ i \in [n]$. \\
 \noindent{\bf{Case (ii):}} $y \in [m] \setminus \{z_1, \ldots, z_n\}, y \in X_{\sigma \setminus <y>}$.\\
 Since $y \in X_{\sigma \setminus <y>}, y \in X_{i}^{\tau}$ for some $i \in [n]$. 
 From $(i)$ and $(ii)$, $y \in X_i^{\tau} \setminus \bigcup \limits_{j \neq i} X_j^{\tau}$ and $y > z_i$.
 We first assume that $\sigma \in S_y \cup \mu_y(S_y)$, {\it i.e.},
 $\sigma \setminus \{<y>\} \notin S_j \cup \mu_j(S_j), 1 \leq j < y$. The set 
 $\{\sigma \, \cup <z_i>\} \setminus \{<y>\}$ has a neighbor 
 $z_1 \ldots z_{i-1} y z_{i+1} \ldots z_n$ and is therefore a simplex  in $\N(G)$.
 \begin{claim}
  $\{\sigma \, \cup <z_i>\} \setminus \{<y>\} \notin S_t \cup \mu_t(S_t), 1 \leq t < z_i$.
 \end{claim}
Suppose $\exists \ i_1 < z_i$ such that $\{\sigma \, \cup <z_i> \} \setminus \{<y>\} \in S_{i_1} \cup \mu_{i_1}(S_{i_1})$.
Since $<y> \ \notin \sigma  \setminus \{<y>\}$,  Proposition \ref{prop1} holds when $\sigma$ is replaced by 
$\sigma \setminus \{<y>\}$ in the definition of $\eta$. Now, using  Proposition \ref{propmain}, 
we get a simplex $\xi_s, 1 \leq s < i_1 \leq m$, where $\xi_s = \{\sigma \, \cup \{<j_1>, \ldots, <j_l>\}\}
\setminus \{<y>, <k_1>, \ldots, <k_t>\}$ if $s$ is odd and $\xi_s = \{\sigma \, \cup \{<z_i>,<j_1>, \ldots, <j_l>\}\}
\setminus \{<y>, <k_1>, \ldots, <k_t>\}$ if $s$ is even, $j_1, \ldots, j_l, k_1, \ldots, k_t < z_i$ such that 
exactly one element in $\{\xi_s, \xi_s \setminus \{<z_i>\}\} \in S_1 \cup \mu_1(S_1)$
(or $\xi_s, \xi_s \, \cup <z_i> \in S_1 \cup \mu_1(S_1)$).
From  Proposition \ref{prop1}, this  is not possible and therefore $\{\sigma \, \cup <z_i>\} \setminus \{<y>\} \notin 
S_t \cup \mu_t(S_t) \ \forall \ 1 \leq t < z_i$. Thus the claim is true.

Since, $\sigma \setminus \{<y>\} 
\notin S_t \cup \mu_t(S_t) \ \forall \ 1 \leq t < z_i$ (by the assumption),  $\sigma \setminus \{<y>\} \in 
S_{z_i}$, a contradiction. Thus $\sigma \notin S_y \cup \mu_y(S_y)$. 
Therefore, $\sigma$ is a critical cell.

 \end{proof}
Combining  Lemmas \ref{theorem1}, \ref{lemma5} and \ref{lemma6}, we have the following necessary and sufficient conditions for 
$\sigma$ to be a critical cell.
\begin{thm} \label{theorem2}
 $\sigma$ is a critical cell if and only if 
  \begin{itemize}
  \item[(i)] $X_{\sigma} = [m] \ \text{or} \ [m]\setminus \{1\}.$
  \item[(ii)]  $<x> \ \in \sigma$ for each $x \in [m] \setminus \{z_1, \ldots, z_n\}$.
  \item[(iii)] $x \in X_i^{\tau} \Rightarrow x \geq z_i \ \forall \ 1 \leq i \leq n$.
 \item[(iv)] $X_r^{\tau} \cap X_s^{\tau} = \emptyset \  \forall \ r \neq s$.
  \item[(v)] there exists $ z < \text{min} \{z_1, \ldots, z_{k-1} ,\widehat{z_k}, z_{k+1}, \ldots,  z_n\}$
  where $z \in X_k^{\tau}$.
 \end{itemize}
\end{thm} 
\noindent{\bf{Case 2.}} $N(\sigma) = \ <1>$.

In this case, for all critical cells $\sigma$, $X_{\sigma} = [m] \setminus \{1\}$, which implies that
$1 \in A_i^{\sigma} \ \forall \ i \in [n]$. Further, since $A_i^{\sigma} \cap A_j^{\sigma} = \{1\}$ for $i \neq j$, 
$A_{i_1}^{\sigma} = A_{i_2}^{\sigma} = \{1\}$ for at least two distinct elements $i_1, i_2$ in $[n]$ ($A_i^{\sigma} = \{1\}$
for only 
one $i \Rightarrow z_1 \ldots z_{i-1} 1 z_{i+1} \ldots z_n \in N(\sigma), z_j \in A_j^{\sigma},
j \neq i$ and $A_i^{\sigma} \neq \{1\} \ \forall \ i \Rightarrow \sigma \in S_1$).
\begin{lem} \label{critical1}
 If $\sigma$ is critical, then $\sigma = \{<2>, \ldots, <m>\}$.
\end{lem}
\begin{proof}
 If $<i> \ \notin \sigma$ for $i \neq 1$, then $\sigma \, \cup <i> \ \subset N(<1>)$ thereby implying that
 $\sigma \, \cup <i> \ \in \N(G)$. 

 The proof is by induction on $i$. If $<2> \ \notin \sigma $, then $\sigma, \sigma \, \cup <2> \ \notin S_1$
 implies that $\sigma \in S_2$, a contradiction.

 If $2 \in X_{\sigma \setminus <2>}$, then $2 \in X_l^{\sigma \setminus <2>}$ for at least one $l \in [n]$. If 
 $2 \in X_l^{\sigma \setminus <2>}$ for exactly one $l \in [n]$, then at least one of $A_{i_1}^{\sigma \setminus <2>}$ or 
 $A_{i_2}^{\sigma \setminus <2>}$ is $\{1\}$ and if $2 \in X_l^{\sigma \setminus <2>} \cap X_k^{\sigma \setminus <2>}$, then 
 $A_{i_1}^{\sigma \setminus <2>} = A_{i_2}^{\sigma \setminus <2>} = \{1\}$. In both these cases 
 $\sigma \setminus <2> \ \notin S_1$ and therefore, $\sigma \in S_2$, a contradiction. Therefore, $<2> \ \in \sigma$ and 
 $2 \notin X_{\sigma \setminus <2>}$. 

 Assume that $<t> \ \in \sigma $ and $t \notin X_{\sigma \setminus <t>} \ \forall \ 1 < t < m$. Suppose 
 $<m> \ \notin \sigma$. Since $t \notin X_{\sigma \setminus <t>} \ \forall \ 1 < t < m$, 
 $<t> \ \in N(\{\sigma \, \cup <m>\} \setminus \{<t>\})$, which implies that
 $\{\sigma \, \cup <m>\} \setminus \{<t>\} \in S_1$. 
 Therefore, $\sigma \, \cup <m> \ \notin S_t \cup \mu_t(S_t) \ \forall \ 1 \leq t < m$, which implies that 
 $\sigma \in S_m$, a contradiction. Hence $<m> \ \in \sigma$.

 If $m \in X_{\sigma \setminus <m>}$, then $m \in \text{Im} \, f$, where $f \in \sigma$ and $|\text{Im}\, f| = n$. 
 By the induction 
 hypothesis, $1, \ldots, m-1 \notin \text{Im}\, f$. Further, since $n > 1$, this case is not possible. Hence, 
 $m \notin X_{\sigma \setminus <m>}$.
 
\end{proof}

We now describe the Morse Complex $\M= (\M_{i}, \partial)$ corresponding to this acyclic matching on the poset $P$.
If $c_i$ denotes the number of critical $i$ cells,
then the free abelian group generated by these critical cells is denoted by $\mathcal{M}_i$.
We use the following version
of Theorem \ref{dmt}, from which we explicitly compute the boundary maps in the Morse Complex $\M$.

\begin{prop} \label{prop6}(Theorem 11.13 \cite{dk})

Let $X$ be a simplicial complex and $\mu$ be an acyclic matching
on the face poset of $X \setminus \{\emptyset\}$. Let $c_i$ denote
the number of critical $i$ cells of $X$. Then
\begin{itemize}
 \item [(a)] $X$ is homotopy equivalent to $X_c$, where $X_c$ is a $CW$ complex with $c_i$ cells in dimension $i$.
 \item [(b)] There is a natural indexing of cells of $X_c$ with the critical cells of $X$ such that for any two cells
 $\tau$ and $\sigma$ of $X_c$ satisfying dim $\tau$ = dim $\sigma + 1$, the incidence number $[\tau: \sigma]$
 is given by
 \begin{center}
  $[\tau:  \sigma ]$ = $\sum\limits_{c} w(c).$
\end{center}
The sum is taken over all (alternating) paths $c$ connecting $\tau$ with $\sigma$
i.e., over all sequences $c$ = $\{\tau, y_1, \mu(y_1 ), \ldots, y_t, \mu(y_t ), \sigma \}$ such that $\tau \succ y_1$,
$\mu(y_t) \succ \sigma$, and $\mu(y_i) \succ y_{i+1}$ for $i = 1, \ldots, t-1$. The quantity $w(c)$ associated to this
alternating
path is defined by
\begin{eqnarray*}
 w(c) := (-1)^t [\tau: \sigma] [\mu(y_t): \sigma] \prod\limits_{i=1}^{t}[ \mu(y_i): y_i]
 \prod\limits_{i=1}^{t-1} [\mu(y_i): y_{i+1}]
\end{eqnarray*}
where all the incidence numbers are taken in the complex $X$.
\end{itemize}
\end{prop}

\subsection{Alternating paths}
 Since we have a description of the critical cells in $\N(G)$, we now study the alternating paths between them.
 We first consider critical $p+1$ cells, which are of the type
 $\sigma = \{<x> | \ x \in T \} \cup \{f_1, f_2\} \subset N(z_1 \ldots z_n)$, where $T = [m] \setminus \{z_1, \ldots, z_n\}$, 
 $| \text{Im} \,f_1| = |\text{Im}\, f_2| = n$ and $z_k = 1$. 

 Let  $\tau = \{f_1, f_2\}$ and the non-empty sets $X_i^{\tau} \cap \, T, 1 \leq i \leq n$ be 
  labelled $Y_{i_1}, \ldots, Y_{i_q}$, where $i_1 \leq i_2 \leq \ldots \leq i_q$.
  By Theorem \ref{theorem2} $(iv)$, $Y_{i_j} \cap Y_{i_k} = \emptyset \ \forall \ j \neq k.$
   For any $l, 1 \leq l \leq q$, let $z = \text{min} \{z_{i_1}, \ldots, z_{i_{l}}\}, D$ 
   any subset of $\{z_{i_1}, \ldots, z_{i_{l}}\}$ and $C \subset \bigcup \limits_{j=1}^{l} Y_{i_j}$ be such that
  $C \cap Y_{i_j} \neq \emptyset \ \forall \ 1 \leq j \leq l$. We first prove the following: 
\begin{lem} \label{lemma11}
The cell
$\alpha = \{ \sigma \cup \{<y>  | \ y \in D \}\} \setminus \{<x>  | \ x \in C \}
  \in S_{z} \cup \mu_z(S_z)$.
  \end{lem}
\begin{proof}

Any element of $C$ is different from $1$ and therefore belongs to $X_{\sigma}.$ For  $j \in \{1, \ldots, l\}$, 
there exists at least one element $y \in C$ such that $y \in A_{i_j}^{\alpha}$ and therefore in $ X_{i_{j}}^{\alpha}$.
If $t \in [n] \setminus \{i_1, \ldots, i_l\}$, then $z_t \in A_{t}^{\alpha}$. Hence, $A_t^{\alpha} \neq \emptyset
\ \forall \ t \in [n]$ and therefore $\alpha \in \N(G).$
\begin{itemize}
\item[(i)] Let $z = 1$. If $t \in [n] \setminus \{i_1, \ldots, i_l\}$, then $z_t \neq 1$ and 
$z_t \in A_t^{\alpha \cup <1>}$,
which implies that  $A_t^{\alpha \cup <1>} \neq \emptyset$. 
Since $T \cap \{1\} = \emptyset$, 
$y \in A_{i_j}^{\alpha}$ implies that $y \in A_{i_j}^{\alpha \cup <1>} \ \forall \ j \in \{1, \ldots, l\}.$
Thus, $\alpha \, \cup <1> \ \in \N(G)$ and therefore $\alpha \in S_1 \cup \mu_1(S_1)$.
\item[(ii)] Let $z > 1.$  Consider the cell
$\eta = \{\sigma \cup \{<a_i>, 1 \leq i \leq s\}\} \setminus \{<b_j>, 1 \leq j \leq r\}$, $s, r \geq 0, 
a_i, b_j < z $.
Let $D' = \{<x>  | \ x \in D \}$,  $C' = \{<y>  | \ y \in C\}$ and  

 \begin{center}
   $\gamma = \begin{cases} \{ \eta \cup D' \} \setminus C' & \text{if} ~ <z> \ \notin \ D' \\
   
              \{\eta \cup D' \} \setminus \{<z>, C'\} & \text{if} ~ <z> \ \in  D',
                       \end{cases}$
 \end{center}

\begin{prop} \label{prop4} $\eta \in \N(G) \Leftrightarrow \gamma , \gamma \, \cup <z> \ \in \N(G)$.
In particular, $\eta \in S_1 \Leftrightarrow \gamma, \gamma \, \cup <z> \ \in S_1$.

\end{prop}
Since $\sigma$ is critical, by  Theorem \ref{theorem2} $(iii)$  $y \in C$ has the property $ y > z$,. If $j \in \{1, \ldots, l\},$ then there exists $y \neq 1, z$ in $C$ such that $y$ belongs to $A_{i_j}^{\gamma}$ and 
therefore  also  to $A_{i_j}^{\gamma \cup <1>}, A_{i_j}^{\gamma \cup <z>}$ and $A_{i_j}^{\gamma \cup <z> \cup <1>}$.
If $1 \leq j \leq l$, $z_{i_j} \geq z > 1$ and $z_{i_j} \in A_{i_j}^{\eta}$ implies that 
$z_{i_j} \in A_{i_j}^{\eta \cup <1>}.$
Let $t \in [n] \setminus \{i_1, \ldots, i_l\}$. $X_{\sigma} = [m]$ or $[m] \setminus \{1\}$,
$j \in \{1, \ldots, l\}$ implies that $z_{i_j} \in X_{i_j}^{\sigma}$. Therefore,
$z_{i_j} \notin A_{t}^{\eta} \ \forall \ t \neq i_j$ and  $D \cap A_{t}^{\eta} = \emptyset.$ 
By the definition of $\gamma, A_{t}^{\eta} = A_t^{\gamma} \ \forall \ t \in [n] \setminus \{i_1, \ldots, i_l\}$.
$A_t^{\gamma \cup <z>} = A_t^{\gamma}$ or $A_t^{\gamma} \setminus \{z\}$.
Now, we only have to consider $t \in [n] \setminus \{i_1, \ldots, i_l\}.$ 

If $\eta \in \N(G),$ then $A_{t}^{\eta} \neq \emptyset $ and $z \notin A_t^{\eta}$. Therefore 
$A_t^{\eta} = A_{t}^{\gamma} = A_t^{\gamma \cup <z>}$ thereby showing that $\gamma, \gamma \, \cup <z> \ \in \N(G)$. 

If $\eta \notin \N(G),$ then for some $t \in [n] \setminus \{i_1, \ldots, i_l\}, A_{t}^{\eta} = \emptyset$  and hence
so are $A_t^{\gamma}$ and $A_t^{\gamma \cup <z>}.$ Thus $\gamma$ and $\gamma \, \cup <z> \ \notin \N(G).$ 

If $\eta \in S_1, A_{t}^{\eta \cup <1>} \neq \emptyset$ and therefore $A_t^{\gamma \cup <1>}$, $A_t^{\gamma \cup <z> \cup <1>}\neq \emptyset$. 

If $\eta \notin S_1$, then for some $t \in [n] \setminus \{i_1, \ldots, i_l\}, A_t^{\eta \cup <1>} = \emptyset$,
thereby implying that $A_t^{\eta} = \{1\}$. Thus, $A_t^{\gamma} = A_t^{\gamma \cup <z>} = \{1\}$ and
the proof of Proposition \ref{prop4} follows.

 Suppose there exists $i_1 < z$ such that   $\alpha \in S_{i_1} \cup \mu_{i_1}(S_{i_1})$.

Clearly, $i_1 > 1$ as  $\alpha \ \in S_{1} \cup \mu_1(S_1)$ implies that $\sigma \in S_1 \cup \mu_1(S_1)$, by Proposition \ref{prop4}, which is a  contradiction.
Define the cell
\begin{center}
   $\xi_1 = \begin{cases}
             \sigma \setminus \{<i_1>\} & \text{if} ~~ <i_1> \ \in \sigma\\
              \sigma \, \cup <i_1> & \text{if} ~~ <i_1> \ \notin \sigma.\\
                       \end{cases}$
 \end{center}
Since  $\alpha  \in S_{i_1} \cup \mu_{i_1}(S_{i_1})$, $\{\xi_{1}  \cup D'\} \setminus C'$ is always a simplex, which implies that  $\xi_{1} \in \N (G)$, by Proposition \ref{prop4}.
 If  $\xi_1 \in S_{i_1} \cup \mu_{i_1}(S_{i_1})$, {\it i.e.} $\xi_{1} \notin S_{j} \cup \mu_{j}(S_{j})$ $\forall \ j <i_1$, then $\sigma$ will belong to $S_{i_1} \, \cup \mu_{i_1}(S_{i_1})$, a contradiction. Therefore, $\exists \ i_2 < i_1$ such that 
 $\xi_1 \in S_{i_2} \cup \mu_{i_2}(S_{i_2})$. Now, $\{\xi_1 \, \cup D'\} \setminus C'$, which is $\alpha \setminus \{<i_1>\}$
 or $\alpha \ \cup <i_1>$, belongs to $S_{i_1} \cup \mu_{i_1}(S_{i_1}).$ If $i_2 =1$, then $\xi_1 \in S_1 \cup \mu_1(S_1)$ and  $\{\xi_1 \, \cup D'\} \setminus C' \notin S_1 \cup \mu_1(S_1)$, which is a  contradiction from Proposition \ref{prop4}.
 Let $i_2 > 1$.
 Define the cell 
 \begin{center}
   $\xi_2 = \begin{cases}
              \{\xi_1 \, \cup D'\} \setminus \{C', <i_2>\} & \text{if} ~~ <i_2> \ \in \sigma\\
              \xi_1 \, \cup D'\  \cup <i_2> & \text{if} ~~ <i_2> \ \notin \sigma.
                       \end{cases}$
 \end{center}

 By  Proposition \ref{prop4}, $\xi_2 \in \N(G)$. Since 
 $\{\xi_1 \, \cup D'\} \setminus C'\ \notin S_{i_2} \cup \mu_{i_2}(S_{i_2})$, there exists $i_3 < i_2$ such that
 $\xi_2 \in S_{i_3}\cup \mu_{i_3}(S_{i_3})$. Here, $\{\xi_2 \cup C'\}\setminus D'$ which is $\xi_1 \setminus \{<i_2>\}$ or
 $\xi_1 \, \cup <i_2> \ \in S_{i_2} \cup \mu_{i_2}(S_{i_2})$. If $i_3 = 1$, then we get a contradiction by Proposition \ref{prop4}. 
 Inductively, assume that there exists $l, 1  <l < m$, where $ i_{l+1} < i_{l} < \ldots < i_1 < z$ and 
 $\xi_{t} \in S_{i_{t+1}}\cup \mu_{i_{t+1}}(S_{i_{t+1}})$, $1 \leq t \leq l$  such that
\begin{center}
   $\xi_{t} = \begin{cases}
              \{\xi_{t-1} \, \cup D'\} \setminus \{C', <i_{t}>\} & \text{if} ~ <i_{t}> \ \in \xi_{t-1},~ t~\text{is}~\text{even}   \\
              \{\xi_{t-1} \, \cup D' \cup <i_{t}>\} \setminus C'& \text{if} ~~ <i_{t}> \ \notin \xi_{t-1},~ t~\text{is}~\text{even}\\
                                     \{\xi_{t-1} \cup C' \} \setminus \{D',<i_{t}>\} & \text{if} ~~ <i_{t}> \ \in \xi_{t-1},~ t~\text{is}~\text{odd}\\
              \{\xi_{t-1} \, \cup C' \cup <i_{t}>\} \setminus D' & \text{if} ~~ <i_{t}> \ \notin \xi_{t-1},~ t~\text{is}~\text{odd}.
                       \end{cases}$
 \end{center}

Since $\{\xi_{l} \cup C'\}\setminus D' ( \text{or} ~~ \{\xi_{l} \cup D' \} \setminus C') =  \xi_{l-1} \setminus \{<i_l>\}$ or 
 $\xi_{l-1} \, \cup <i_l> \ \in S_{i_l} \cup \mu_{i_l}(S_{i_l})$ and $i_l > i_{l+1}$, if $i_{l+1} = 1$, using Proposition \ref{prop4} we arrive at a contradiction.
 
By induction, there exists $1 < s \leq m$, and a cell 
 \begin{center}
   $\xi_s = \sigma \cup \begin{cases}  \{D',  \langle b_1 \rangle , \ldots, \langle b_r \rangle \}
 \setminus \{ C', \langle l_1 \rangle, \ldots, \langle l_t \rangle\} & \text{if} ~ \text{s}~ \text{is} ~  \text{even} \\
               \{<b_1>, \ldots, <b_r>\}
 \setminus \{ \langle l_1 \rangle , \ldots, \langle l_t \rangle \}& \text{if} ~ \text{s} ~ \text{is} ~  \text{odd},
                       \end{cases}$
 \end{center}

 where $r, t \geq 0$ and $b_1, \ldots, b_r, l_1, \ldots, l_t < z$ such that,  $\xi_{s} \in S_{1}\cup \mu_{1}(S_1)$ but
 $\{\xi_s \cup C'\} \setminus D'$ (if  $s$ is even) or $ \{ \xi_s \cup D' \} \setminus  C'$ (if $s$ is odd)
 $\notin S_1 \cup \mu_1(S_1)$. But, this  is not possible by Proposition \ref{prop4}.

Hence $\alpha \notin S_i \cup \mu_i(S_i) \ \forall \ i < z$. 
Replacing $\alpha$ by $\alpha \, \cup <z>$ (if $<z> \ \notin \alpha$) or by $\alpha \setminus \{<z>\}$ (if $<z> \ \in \alpha$)
and applying an argument similar to the one above, we see that  $\alpha \, \cup <z>$
(if $<z> \ \notin \alpha$) or $\alpha \setminus \{<z>\}$
(if $<z> \ \in \alpha$) $\notin S_i \cup \mu_i(S_i) \ \forall \ i < z.$
 Hence $\alpha \in S_{z} \cup \mu_z(S_z)$. 

 \end{itemize}
\end{proof}

\begin{lem} \label{lemma12}
 Let $\eta = \{<x_1>, \ldots, <x_p>, f_1, f_2\} \subset N(w_1 \ldots w_n)$ be a simplex,
 where $ w_k = 1, w_{i_0} = \text{min} 
 \{w_1, \ldots,\widehat{w_k}, \ldots, w_n\}$, $f_{2}(i) = w_i \ \forall \ i \neq k$ and
 $ [m] \setminus \{ w_1, \ldots, w_n\} = \{x_1, \ldots, x_p\}$.
 If $f_2(k) > w_{i_0}$, then $\eta \setminus \{f_1\} \in S_{w_{i_0}}$. 
\end{lem}
\begin{proof}
$f_2(k) > w_{i_0} > 1$ implies that $f_2(k) \in \{x_1, \ldots, x_p\}$ and $1 \notin \text{Im}\,f_2$. Therefore, 
$\{\eta \, \cup <w_{i_0}>\} \setminus \{f_1\} \sim \ <1>$. If $y_1 \ldots y_n \sim \eta \setminus \{f_1\}$, 
then $y_k \notin \{w_1, \ldots, \widehat{w_k}, \ldots , w_n, x_1, \ldots, x_p\}$. 
Thus $y_k =1$ and $\eta \setminus \{f_1\} \notin S_1$. Consider $j$ such that $1 < j < w_{i_0}.$ Since 
$j \in \{x_1, \ldots, x_p\}$ and $j \notin X_{\eta \setminus \{f_1\}}$, $\eta \setminus \{f_1, <j>\}$ and 
$\{\eta \,\cup <w_{i_0}>\} \setminus \{f_1, <j>\} \sim \ <j>$, thereby showing that $\eta \setminus \{f_1\}$ and 
$\{\eta \, \cup <w_{i_0}>\} \setminus \{f_1\} \notin S_{j} \cup \mu_j(S_j) \ \forall \ j < w_{i_0}$. Thus 
$\eta \setminus \{f_1\} \in S_{w_{i_0}}$.

  \end{proof}

\begin{lem} \label{lemma13}
No element of  $S_1 \cup \mu_1(S_1)$  belongs to any alternating path between two critical cells.
\end{lem}
\begin{proof} 
 Let $c$ = $\{\tau, y_1, \mu(y_1), \ldots, y_t,
\mu(y_t),\alpha \}$ be an alternating path between the critical cells $\tau$ and $\alpha$. 
Let $S = \bigcup \limits_{i=1}^{m}S_i$ and $ \mu : S \rightarrow P \setminus S$ be the map such that 
$\mu|_{S_i} = \mu_i$. Here, $y_i \in S$, $\mu(y_i) \in \mu(S)$ and $S \cap \mu(S) = \emptyset.$
If $\gamma \in c$ and  $\gamma \in \mu_1(S_1)$, then 
$\gamma = \mu(y_i)$ for some $i \in [t]$. Since $[\mu(y_i): y_i] = \pm 1$, $y_i = \mu(y_i) \setminus \{<1>\}$.
In the alternating path,
$y_{i+1} \prec \mu(y_i)$ and $y_{i+1} \neq y_i$ implies that $<1> \ \in \mu(y_i) \setminus \{<1>\}$.
This implies that $y_{i+1} \in \mu(S)$, a contradiction. Therefore $\gamma \notin c.$ If $\gamma \in S_1$, then 
$\mu(\gamma) \in c$, a contradiction.
 \end{proof}
We are now ready to prove the main result of this section.
\begin{thm} \label{theorem3}
  There exist two critical $p$-cells $\beta$ and $\gamma$ such that there exists
  exactly one alternating path from $\sigma$  to each of $\beta$ and $\gamma$.
  Further, there exists no alternating path from $\sigma$ to any other critical $p$-cell.
\end{thm}
\begin{proof}
Define the sets $A_i, B_i, i \in \{1, 2\}$ to be 
$A_i = T \cap \{\text{Im} f_i \setminus \{f_{i}(k)\}\}$ and
$B_i = \{z_1, \ldots,\widehat{z_k}, \ldots,  z_n\} \setminus \text{Im} \, f_i$.
Since $A_i \subset T$ and $B_i \subset \{z_1, \ldots, z_n\}, A_i \cap B_i = \emptyset \ \text{for} \ 1 \leq i \leq 2.$
If $z_i \notin \{f_1(i), f_2(i)\}$ for $i \neq k$,
then $z_i \notin \text{Im} \, f_1, \text{Im} \, f_2, T$ and so $z_i \notin X_{\sigma}$.
Since $z_i \neq 1, <z_i> \ \sim \sigma , <1>$ and thus $\sigma \in S_1$. Therefore, $z_i \in \{f_1(i), f_2(i)\} \ \forall \ 
i \in [n] \setminus \{k\}$, thus implying that $B_1 \cap B_2 = \emptyset, B_1 \subset \text{Im} \, f_2$ and $B_2 \subset 
\text{Im} \, f_2$.  

Let $x \in A_1 \cap A_2$. For $i \in [n] \setminus \{k\}, z_i \in \{f_1(i), f_2(i)\}$ and therefore,
$\exists \ i \neq j \in [n]$ such that $f_1(i) = f_2(j) = x$. Here, $X_i^{\tau} \cap X_j^{\tau} \neq \emptyset$, 
a contradiction
to  \ref{theorem2} $(iv)$ and hence $A_1 \cap A_2 = \emptyset.$

Let $W = \{z \ | \ z = f_{1}(i) = f_2(i)\}$
and $T_1 = T \setminus \{A_1 \cup A_2\}$. 
Given a set $B \subset B_1 \cup B_2$, define a corresponding set
$A \subset A_1 \cup A_2$ as follows. 
If $z_i \in B \cap B_1$, then 
$z_i  \neq f_1(i)$ and hence, $f_1(i) \in A$. For $z_i \in B \cap B_2$,
$f_2(i) \in A$. The number of elements in $A$ and $B$ are the same.

Let the facet of $\sigma$ in the alternating path $c$ be $y_1$. We consider the following two cases.
\vspace{0.3 cm}

\noindent{\bf{Case 1.}} $A_1, A_2 = \emptyset$. \\
 In this case, $f_1(i) = f_2(i) = z_i \ \forall \ i \neq k$ and $\{ f_1(k), f_2(k)\} \cap T \neq \emptyset.$ 
If an element $x$ of $T$ does not belong to $X_{\tau}$, then $\sigma \setminus \{<x>\} \sim \ <x>$ and so 
$\sigma \setminus \{<x> \} \in S_1.$
 If $x \in X_{\tau}$, then $x \in \{f_1(k), f_2(k)\}$
  and $\sigma \setminus \{<x>\}  \subset N(z_1 \ldots z_{k-1} x z_{k+1} \ldots z_n)$
 which implies $\sigma \setminus \{<x>\} \in S_1$ and from  \ref{lemma13},
 $\sigma \setminus \{<x>\} \notin c$. Therefore, $y_1$ is either $\sigma \setminus \{f_1\}$ or $\sigma \setminus \{f_2\}$.
 Each of  $X_{\sigma \setminus \{f_1\}}$ and
 $X_{\sigma \setminus \{f_2\}}$ will be either $[m]$ or $[m] \setminus \{1\}.$ 
 Further, $\sigma \setminus \{f_1\}$ and $\sigma \setminus \{f_2\}$ satisfy the properties $(ii), (iii)$ and $(iv)$
 of Theorem  \ref{theorem2}. By Theorem \ref{theorem2} $(v)$,
 at least one of $f_1(k) \ \text{or} \ f_2(k) < z_{i_0}$, where $z_{i_0} = \text{min}\{z_1, \ldots, \widehat{z_k}, 
 \ldots, z_n\}$.
 
 If $f_1(k)  < z_{i_0}$, then $\sigma \setminus \{f_2\}$ is a critical $p$-cell as it satisfies all the five
 conditions of Theorem  \ref{theorem2} and $c = \{\sigma , \sigma \setminus \{f_2\}\}$.\\
 Similarly, $f_2(k) < z_{i_0}$ implies $\sigma \setminus \{f_1\}$  is critical and 
 $c = \{\sigma , \sigma \setminus \{f_1\}\}$.
 
Now, assume that $f_1(k) < z_{i_0}, f_2(k) > z_{i_0}$ and $y_1 = \sigma \setminus \{f_1\}.$
Since $f_2(i) = z_i \ \forall \ i \neq k$ and $f_2(k) > z_{i_0}, \sigma \setminus \{f_1\} \in S_{z_{i_0}}$
by  Lemma \ref{lemma12}. So $\mu(y_1) = \{\sigma \, \cup <z_{i_0}> \} \setminus \{f_1\}$. 
If $x \in [m] \setminus \{z_1, \ldots, z_n , f_2(k)\}$, then $\mu(y_1) \setminus \{<x>\}  \sim \ <x>$ and
hence belongs to $S_1$. $\mu(y_1) \setminus \{<x> \} = y_1$ if $x = z_{i_0}$.
$\{\sigma \, \cup <z_{i_0}>\} \setminus \{f_1, f_2\} \sim \ <z_i>$ for $i \neq i_0, k$
and so $y_2$ has to be $\mu(y_1) \setminus \{<x>\}$ where $x = f_2(k)$. 
By  Theorem \ref{theorem2},
$y_2 \subset N(\tilde{z_1}\ldots \tilde{z_n})$, 
where $\tilde{z}_k = f_{2}(k), \tilde{z}_{i_0} = 1$ and $\tilde{z}_{i} = z_i, i \neq i_0, k$,
is a critical cell and the alternating path in this case is
$\{\sigma, \sigma \setminus \{f_1\}, \{\sigma \, \cup <z_{i_0}>\} \setminus \{f_1\}, 
\{\sigma \, \cup <z_{i_0}> \} \setminus \{f_1, <f_2(k)>\} \}.$

If $f_2(k) < z_{i_0}, f_1(k) > z_{i_0}$ and $y_1 = \sigma \setminus \{f_2\}$, then the alternating path is
$\{\sigma, \sigma \setminus \{f_2\}, \{\sigma \, \cup <z_{i_0}>\} \setminus \{f_2\}, \{\sigma \, \cup <z_{i_0}> \} 
\setminus \{f_2, <f_1(k)>\} \}.$
\vspace{0.3 cm}

\noindent{\bf{Case 2.}} At least one of $A_1$ or $A_2$ is non empty.\\
We  prove some results necessary for the construction of the  alternating paths $c$ from $\sigma$.
\begin{lem} \label{lemma15}
 Let $\alpha = \{\sigma \cup \{<b> | \ b \in B\}\} \setminus \{<a> | \ a \in A\}$, where $B \neq \emptyset$. Then, 
 \begin{itemize}
  \item[(i)] $\alpha \setminus \{<x>\} \in S_1 \ \forall \ x \in T_1.$
  \item[(ii)] If $B_1 \cap B \neq \emptyset, B_1$ (or $B_2 \cap B \neq \emptyset, B_2$),
  then $\alpha \setminus \{f_1\}, \alpha \setminus \{f_2\} \in S_1$.
  \item[(iii)] If $B = B_1 \cup B_2$ and $B_1, B_2 \neq \emptyset$,
  then $\alpha \setminus \{f_1\}, \alpha \setminus \{f_2\} \in S_1$.
  \item[(iv)] If $B = B_1$ (or $B = B_2$), then 
  $\alpha \setminus \{f_1\} \in S_1$ and $\alpha \setminus \{f_2\} \notin S_1$ (or $\alpha \setminus \{f_1\} \notin S_1$ and 
  $\alpha \setminus \{f_2\} \in S_1$.)
 \end{itemize}

\end{lem}
\begin{proof} \hspace{10 cm}
 \begin{itemize} 
  \item [(i)] Let $x \in T_1$. If $x \notin X_{\tau},$ then $x \notin X_{\alpha \setminus <x>}$ and therefore $
  \alpha \setminus \{<x>\} \sim \ <x>$, which implies  $\alpha \setminus \{<x>\} \in S_1.$
If $x \in X_{\tau}$, then $x \in \{f_1(k), f_2(k)\}$
  and $\alpha \setminus \{<x>\}  \sim w_1 \ldots w_{k-1} x w_{k+1} \ldots w_n$, where $ \alpha \sim w_1w_2\dots w_n$,
 which implies that $\alpha \setminus \{<x>\} \in S_1$.

 \item[(ii)] Since $B \cap B_1 \neq B_1$, $\exists \ b \in B_1 \setminus B$ such that $<b> \ \notin \alpha, 
 b \in \text{Im} \, f_2$ and $b \notin \text{Im} \, f_1$. Thus $\alpha \setminus \{f_2\} \sim \ <b>$.
 $A \cap A_1 \neq \emptyset$ because $B \cap B_1 \neq \emptyset$. So if 
 $x \in A \cap A_1, <x> \ \notin \alpha $ and $<x> \ \sim  f_2$. Hence, $\alpha \setminus \{f_1\} \sim \ <x>$.

 \item[(iii)] If $B = B_1 \cup B_2$, then $ A = A_1 \cup A_2$. Thus, $\alpha \setminus \{f_1\} \sim \ <x>$, where 
 $x \in A_1$ and $\alpha \setminus \{f_2\} \sim \ <y>$, where $y \in A_2$.

 \item[(iv)] Let $B = B_1$ and $B \cap B_2 = \emptyset$.

 As in the earlier cases, $\alpha \setminus \{f_1\} \sim \ <x>$ for any $x \in A_1$ 
 and hence $\alpha \setminus \{f_1\} \in S_1$.

 If $z \in B_1$, then 
 $<z> \ \in \alpha$ and so $<z>$ is not a neighbor of $\alpha \setminus \{f_2\}$.
 If $z \in B_2$ or $W$, then $z \in \text{Im} f_1$ and so $<z> \ \nsim f_1$ and thus $<z> \ \nsim \alpha \setminus \{f_2\}.$
 If $x \in A_1$, then $x \in \text{Im}f_1$ and $<x> \ \nsim f_1$.

 If $x \in A_2$, then $<x> \ \in \alpha$. Thus, the only possible neighbors of $\alpha \setminus \{f_2\}$ are 
 $<1>$ and $y_1 \ldots y_n, y_i \neq y_j$ for $i \neq j$.

Let $\alpha \setminus \{f_2\} \sim y_1 \ldots y_n = y$.
For $z \neq 1 \in B_2 \cup W \cup A_1$, since $y \sim f_1$, we see that $y_k \neq z$. 
If $x \in T_1 \cup A_2 \cup B_1$, then $<x> \ \in  \alpha \setminus \{f_2\}$ and $y \sim \ <x>$ which implies $y_k \neq x$.
The only possible choice for $y_k$ is $1$. Therefore, $\alpha \setminus \{f_2\} \notin S_1$.
  \end{itemize}
\end{proof}

To construct the alternating path from the critical $p+1$-cell $\sigma$ to a critical $p$-cell in $\N(G)$,
we first require the facet $y_1$ of $\sigma$ in $c$.

If $A_1 = \emptyset$ and $A_2 \neq \emptyset$, then $f_1 = z_1 \ldots z_{k-1} f_1(k) z_{k+1} \ldots z_n$. 
Hence, $\sigma \setminus \{f_1\} \sim \ <b> \ \forall \ b \in B_2$ and $\sigma \setminus \{<x>\} \sim \ <x>
\ \forall \ x \in T_1 \setminus \{f_1(k), f_2(k)\}$. For $x \in T \cap \{f_1(k), f_2(k)\}$, we observe that $x \neq 1$
and $\sigma \setminus \{<x>\} \sim z_1 \ldots z_{k-1} x z_{k+1} \ldots z_n \sim \ <1>$ and thus
$\sigma \setminus \{<x>\} \in S_1$.
Therefore $y_1 = \sigma \setminus \{f_2\} $ or $\sigma \setminus \{<x>\}$, where $x \in A_2$. \\
Similarly, in the case $A_1 \neq \emptyset$ and $A_2 = \emptyset$, we see that $y_1 = \sigma \setminus \{f_1\}$ or 
$\sigma \setminus \{<x>\}, x \in A_1$.

Let $i \neq j \in \{1, 2\}$. Consider the case when $y_1 = \sigma \setminus \{f_i\}$. Let $z_{i_0} = \text{min}
\{z_1, \ldots, \widehat{z_k}, \ldots,$ $  z_n\}$. If $A_i = \emptyset, A_j \neq \emptyset$, then
$y_1 = \sigma \setminus \{f_j\}$. Since $f_i(k) \neq z_l \ \forall \ l \in [n] \setminus \{k\}, f_i(k) \neq z_{i_0}$ and thus 
$f_i(k) > z_{i_0}$ or $f_i(k) < z_{i_0}$.
\vspace{0.3 cm}

\noindent{\bf{Subcase 1.}} $f_i(k) < z_{i_0}$.

Since $f_i = z_1 \ldots z_{k-1} f_i(k) z_{k+1} \ldots z_n, \sigma \setminus \{f_j\} $ satisfies all the criteria of  
Theorem \ref{theorem2} and thus  a critical $p$-cell, {\it i.e.} $c = \{\sigma, \sigma \setminus \{f_j\}\}$.
\vspace{0.3 cm}

\noindent{\bf{Subcase 2.}} $f_i(k) > z_{i_0}$.

By  Lemma \ref{lemma12}, $\mu(y_1) =  \{\sigma \, \cup <z_{i_0}>\} \setminus \{f_j\}$. If $<x> \ \in \mu(y_1)$, 
$x \in T_1 \cup A_j \cup \{z_{i_0}\}$. Since $y_2 \neq y_1, y_2 \neq \mu(y_1) \setminus \{<z_{i_0}>\}$. If $x \in T_1 \cup A_j 
\setminus \{f_i(k)\}$, then $\mu(y_1) \setminus \{<x>\} \sim \ <x>$ and $\mu(y_1) \setminus \{f_i\} \sim \ <z_l>, l \neq i_0, k$.
Therefore, $y_2 = \mu(y_1) \setminus \{<f_i(k)>\}$ and $y_2 \sim w_1 \ldots w_n$, where $w_k = f_i(k), w_{i_0} = 1$ and 
$\forall \ l \neq i_0, k, w_l = z_l$. By Theorem \ref{theorem2}, $y_2$ is a critical cell and 
$c = \{\sigma, \sigma \setminus \{f_j\}, \{\sigma \, \cup <z_{i_0}>\} \setminus \{f_j\}, \{\sigma \, \cup <z_{i_0}>\} \setminus
\{f_j, <f_i(k)>\}\}$.
\vspace{0.3 cm}

If $A_1, A_2$ are both non empty, then $\sigma \setminus \{f_i\} \sim \ <b>  \forall \ b \in B_j$. 
If $x \in T_1 \setminus \{f_1(k), f_2(k)\}$, then $\sigma \setminus \{<x>\} \sim \ <x>$ and 
$\sigma \setminus \{<x>\} \sim z_1 \ldots z_{k-1} x z_{k+1} \ldots z_n$ for $x \in \{f_1(k), f_2(k)\}$. Therefore,
$y_1 = \sigma \setminus \{<x>\}$, where $x \in A_1 \cup A_2$. 

Let $b \in B \subset B_1 \cup B_2$. If $b = z_l \in B_i$, then $f_i(l) \neq z_l, f_j(l) = z_l$ and $f_i(l) \in A_i$. The element 
$f_i(l)$ in $A_i$ is called the element corresponding to $b$ and is denoted by $x_b$. Similarly, if $x \in A_1 \cup A_2$,
then $b_x$ is called the element of $B$ corresponding to $x$.\\
Let $B_1 = \{b_{11} < b_{12} < \ldots < b_{1r}\}, B_2 = \{b_{21} < b_{22} < \ldots < b_{2l}\}$
 and $B = \{b_1 < b_2 < \ldots < b_q\}$ be ordered sets and the corresponding sets $A_1, A_2$ and $A$ be $\{x_{11}, x_{12},
 \ldots, x_{1r}\}$, $\{x_{21}, \ldots, x_{2l}\}$ and $\{x_1, \ldots, x_q\}$ respectively {\it i.e.} $x_{1i} = x_{b_{1i}},
 x_{2j} = x_{b_{2j}}$ and $x_t = x_{b_t} \ \forall \  1 \leq i \leq r, 1 \leq j \leq l, 1 \leq t \leq q$.
 Let $\alpha = \{\sigma \, \cup \{<b> | \ b \in B\}\} \setminus \{<x> | \ x \in A\}$. By Lemma \ref{lemma11}, 
 $\alpha \in \mu_{b_1}(S_{b_1})$. Here, 
 $B$ and  at least one of $B_1$ or $B_2$ are always non empty sets. Let $b_0 = \text{min}\{b \ | \ b \in B_1 \cup B_2\}$.
 \begin{lem} \label{lemma16}
  Let $B \neq B_1, B_2, \eta = \{\sigma \, \cup \{<b_i> | \ 2 \leq i \leq q\} \} \setminus \{<x_i> | \ 1 \leq i \leq q\} \in S_{b_1}$ and $\alpha = \mu_{b_1}(\eta) = \eta \ \cup <b_1> \ \in \mu_{b_1}(S_{b_1})$. If 
  \begin{itemize}
  \item[(i)] $\alpha \in c$ and $b_0 \notin B$, then the only facets  $y$ of $\alpha$ which can 
  belong to $c$ are of the form $\alpha \setminus \{<x>\}$, where $b_x < b_1$ and $x \in \{A_1 \cup A_2\} \setminus A$.
  \item[(ii)] $b_0 = b_1$, then no facet of $\alpha$ belongs to $c$ and thus $\alpha \notin c$. 
  \end{itemize}
  
 \end{lem}
\begin{proof}
By  Lemma \ref{lemma15}, $\forall \ x \in T_1, \alpha \setminus \{<x>\}, \alpha \setminus \{f_1\}$
and $\alpha \setminus \{f_2\}$ belong to $S_1$.
For $b \in B \setminus \{b_1\}$, by Lemma \ref{lemma11}, $\alpha \setminus \{<b>\} \in \mu_{b_1}(S_{b_1})$ and
$\alpha \setminus \{<b_1>\} = \eta$. Since $\alpha \in c$ implies that $\eta \in c$,
the facet $y \neq \alpha \setminus \{<b_1>\}$. For $x \in A_1 \cup A_2$ such that $b_x > b_1$, by  Lemma \ref{lemma11},
$\alpha \setminus \{<x>\} \in \mu_{b_1}(S_{b_1})$. Thus, the only possible facets of $\alpha$ which can belong to $c$
 are of the type $\alpha \setminus \{<x>\}$, where $b_x < b_1$.

 If $b_0 = b_1$, then no facet of $\alpha$ can belong to $c$. Since $\alpha \in \mu_{b_1}(S_{b_1})$, it is not a critical cell and thus $\alpha \notin c$.

\end{proof}

From the above Lemma, we conclude that no cell
$\alpha = \{\sigma \ \cup \{<b> | \  b \in B\}\} \setminus \{<x> | \ x \in A\}$
belongs to $c$, if $B \cap B_1$ and $B \cap B_2$ are both non empty. This follows by an inductive argument.
If $b_1 \neq b_0$, then by Lemma \ref{lemma16} $(i)$, 
$\alpha_1 = \{\alpha \ \cup <b_x>\} \setminus \{<x>\} \in c$, where $b_x < b_1$,
 if $\alpha \in c$. Inductively $\alpha_t = \{\alpha_{t-1} \cup <b_0>\} \setminus \{<x_{b_0}>\} \in c$ which contradicts 
Lemma \ref{lemma16} $(ii)$.
Similarly if $b_0 \in B \cap B_i$ and $B \subsetneq B_i$ for $i \in \{1,2\}$, then
 $\alpha \notin c$. We now show that if $b_{11}$ or $b_{21} \in B$, then $B = B_1$ or $B_2$.

\begin{lem} \label{lemma20}
Let  $\alpha = \{\sigma \, \cup \{<b>| \ b \in B\} \} \setminus \{<x> | \ x \in A\}$ and $b_{11} \in B$. 
If $\alpha \in c$, then $B = B_1$.

\end{lem}
\begin{proof}
By Lemma \ref{lemma11}, $\alpha \in \mu_{b_{11}}(S_{b_{11}}) $ and is thus not a critical cell.
If $B \subsetneq B_1$, then by Lemma \ref{lemma16} $(i)$, the facet of $\alpha \in c$ is of the form $\alpha \setminus \{<x>\}$,
where $b_x < b_{11}$. Since $b_{11}$ is the least element of $B_1, b_x \in B_2$, a contradiction. Thus $B = B_1$.

\end{proof}

Hence, we can conclude that $y_1 = \sigma \setminus \{<x_{1r}>\}$ or $\sigma \setminus \{<x_{2l}>\}$.
\vspace{0.3 cm}

\noindent{\bf{Subcase 3.}} $y_1 = \sigma \setminus \{<x_{1r}>\}$.

Using Lemmas \ref{lemma16} and \ref{lemma20}, 
$y_i = \{\sigma \, \cup \{<b_{1t}> | \ r-i+1 < t \leq r\}\}
\setminus \{<x_{1r}>, \ldots, <x_{1r-i+1}>\}, 1 \leq i \leq r$ and $\mu(y_r) = \sigma \cup \{<b> | \ b \in B_1\} \setminus \{<a> |
\ a \in A_1\}$. From  Lemma \ref{lemma15} and Lemma \ref{lemma16}, $y_{r+1} = \mu(y_r) \setminus \{f_2\}$ or $\mu(y_r) 
\setminus \{<x>\}, x \in A_2$ and $b_x < b_{11}$.
If $y_{r+1} = \mu(y_r) \setminus \{<x>\}$, then 
$\mu(y_{r+1}) = y_{r+1} \cup <b_x>, B \cap B_1, B \cap B_2 \neq \emptyset$ and therefore $\mu(y_{r+1}) \notin c$. 
Thus, $y_{r+1} = \mu(y_r) \setminus \{f_2\}$. 
The vertex $w_1 \ldots w_n$ defined by $w_i = f_1(i) \ \forall \ i \neq k$  and $w_k = 1$ is a neighbor of $y_{r+1}$. Let 
$w_{i_0} =\text{min}\{w_1, \ldots, \widehat{w_k}, \ldots, w_n\}$.
\begin{enumerate}
\item[(a)] $f_1(k) < w_{i_0}$.

$y_{r+1} = \mu(y_r) \setminus \{f_2\}$ satisfies all the properties of Theorem \ref{theorem2}
and is hence a critical
$p$-cell. The alternating path $c = \{\sigma, \sigma \setminus \{<x_{1r}>\}, \{\sigma \, \cup <b_{1r}>\} \setminus \{<x_{1r}>\}, \ldots, 
y_r, \mu(y_r), \mu(y_r) \setminus \{f_2\}\}$.

\item[(b)] $f_1(k) > w_{i_0}$.

Since $f_1(i) = w_i \ \forall \ i \neq k, f_1(k) > w_{i_0}$, by Lemma  \ref{lemma12},
$y_{r+1} = \mu(y_r) \setminus \{f_2\} \in S_{w_{i_0}}$, {\it i.e.} $\mu(y_{r+1}) = \{\mu(y_r) \, \cup <w_{i_0}>\} \setminus \{f_2\}$
($<w_{i_0}> \ \notin \mu(y_r)$, since $w_{i_0} = z_i \Rightarrow w_{i_0} \in B_2$ and $<w_{i_0}> \ \notin \mu(y_r)$ and 
$w_{i_0} \in A_1 \Rightarrow <w_{i_0}> \ \notin \mu(y_r)$).
\begin{claim}
 $y_{r+2} = \mu(y_{r+1}) \setminus \{<f_1(k)>\}$.
\end{claim}
If $<x> \ \in \mu(y_{r+1})$, then $x \in T_1 \cup B_1 \cup A_2 \cup \{w_{i_0}\}$. If $x \in T_1 \cup B_1 \cup A_2$ and
$x \neq f_1(k)$,
then $x \notin \text{Im} \, f_1$ and hence $\mu(y_{r+1}) \setminus \{<x>\} \in S_1$. Further, 
$\mu(y_{r+1}) \setminus \{f_1\} \sim \ <x>$ for any $x \in A_1$. So, $y_{r+2} = \mu(y_{r+1}) \setminus \{<f_1(k)>\}$.
The claim is thus proved.

$y_{r+2} \subset N(v_1 \ldots v_n)$, where $v_{i_0} = 1, v_{k} = f_1(k)$ and $v_i = w_i \ \forall \ i \neq i_0, k$,
satisfies all the conditions of Theorem \ref{theorem2} and is thus a critical $p$-cell. The alternating path 
$c = \{\sigma, \sigma \setminus \{<x_{1r}>\}, \{\sigma \, \cup <b_{1r} \} \setminus \{<x_{1r}>\}, \ldots, 
\{\sigma \, \cup \{<b_{11}>, \ldots, <b_{1r}>\}\} \setminus \{<x_{11}>, \ldots, <x_{1r}>\}, 
\mu(y_{r}) \setminus \{f_2\}, \{\mu(y_r) \, \cup <w_{i_0}> \} \setminus \{f_2\},\{ \mu(y_{r}) \, \cup <w_{i_0}>\} \setminus \{f_2, <f_1(k)>\}
\}$.
\end{enumerate}
\vspace{0.3 cm}

\noindent{\bf{Subcase 4.}} $y_1 = \sigma \setminus \{<x_{2l}>\}$.

Here, $y_i = \{\sigma \, \cup \{<b_{2t}>| \ l-i+1 < t \leq l\}\}
\setminus \{<x_{2l}>, \ldots, <x_{2l-i+1}>\}, 1 \leq i \leq l$ and $y_{l+1} = \mu(y_l) \setminus \{f_1\}$. The vertex
$ u_1 \ldots u_n$, where $u_i = f_2(i), i \neq k$ and $u_k = 1$ is a neighbor of $y_{l+1}$. 
Let $u= \text{min}\{u_1, \ldots, u_n\} \setminus \{1\}$.

\begin{enumerate}
\item[(a)] $f_2(k) < u$.

The alternating path $c$ is $\{ \sigma, \sigma \setminus \{<x_{2l}>\}, \ldots, \mu(y_l) = 
\sigma \cup \{<b> | \ b \in B_2\} \setminus \{<a> | \ a \in A_2\}, \mu(y_l) \setminus \{f_1\}\}.$

\item[(b)] $f_2(k) > u$.

In this case,  $c = \{ \sigma, \sigma \setminus \{<x_{2l}>\}, \ldots, \mu(y_l), \mu(y_l)  \setminus \{f_1\},
\{\mu(y_l)  \, \cup <u> \}\setminus \{f_1\}, \{\mu(y_l)$ $ \, \cup <u>\} \setminus \{f_1, <f_2(k)>\}\}.$

\end{enumerate}
\end{proof}
Consider the critical cell $\tau = \{<2>, <3>, \ldots, <m>\} \subset N(<1>)$. Since every facet of $\tau$ is in $S_1$, 
there exists no alternating path from $\tau$.


Our objective now is to first study the $\Z_2$ homology groups of the Morse complex corresponding to the acyclic matching 
$\mu$ on $P$.

 Let the Discrete Morse Complex corresponding to the acyclic matching $\mu$  be
  $\M = (\M_n , \partial_n)$, $n \geq 0$ where $\M_i$ denotes the free abelian groups over $\mathbb{Z}_2$ generated
  by the critical $i$-cells. 
  For any two critical cells $\tau$ and $\sigma$ such that dim($\tau$) = dim($\sigma)$ +1, the incidence number
$[\tau: \sigma]$ is either 0 or 1.

Let $C_i$ denote the set of critical cells of dimension $i$. Since $<1>$ is the only $0$-dimensional critical cell,
$C_0 = \{<1>\}$. If $n \geq 3$, using Theorem \ref{theorem2} and Lemma \ref{critical1}, we can conclude that 
$C_i = \emptyset$ for $0 < i \leq p-1$. Let $C_{p} = \{\alpha_1, \ldots, \alpha_{r_1}\}$ and 
$C_{p+1} = \{\tau_1, \ldots, \tau_{r_2}\}$. Let
$A = [a_{ij}]$ be a matrix of order  $|C_{p}| \times |C_{p+1}|$, where
$a_{ij} = 1$, if there exists an alternating path from $\tau_j$ to $\alpha_i$ and $0$
if no such path exists.  Using Theorem \ref{theorem3}, each column of $A$ contains exactly two
 non zero elements which are $1$ (except, when $\{<2>, \ldots, <m>\} \in C_{p+1}$ 
 and in this case the column of $A$ corresponding to this cell is zero,
 as there is no alternating path from $\{<2>, \ldots, <m>\}$ to any critical cell).

\begin{thm} \label{main2}
 Let $m-n = p \geq 1$. Then $H_p(\N(G) ; \mathbb{Z}_2) \neq 0.$
\end{thm}
\begin{proof}
$\M_p  \cong \mathbb{Z}_2^{|C_p|}$ and $\M_{p+1} \cong  \mathbb{Z}_2^{|C_{p+1}|}$.
Since each column of $A$ is either zero (when $\{<2>, \ldots, <m>\} \in C_{p+1}$) or
contains exactly two non zero elements, both being $1$, the column sum is zero (mod $2$). Therefore, 
rank($A$) $< |C_{p}|$. In particular, rank of the boundary map $\partial_{p+1} : \mathbb{Z}_2^{|C_{p+1}|} \rightarrow 
\mathbb{Z}_2^{|C_{p}|}$ is strictly less than $|C_{p}|$. 

Since $\text{Hom}(K_2 \times K_n, K_m) \simeq \text{Hom}(K_2, K_m^{K_n}) 
\simeq \N(K_m^{K_n}) \simeq \N(G)$ and the maximum degree of $K_2 \times K_n$ is
$n-1$, $\text{conn}(\text{Hom}(K_2 \times K_n, K_m)) = \text{conn} (\N(G)) \geq m-n-1$. Hence $\N(G)$ is path connected
and therefore $H_0 (\N(G); \mathbb{Z}_2) \cong \mathbb{Z}_2$.

If $p= 1$, then since $\M_0 \cong \mathbb{Z}_2$, Ker($\partial_1$) $\cong \mathbb{Z}_2^{|C_1|}$, where $\partial_1: 
\M_1 \longrightarrow \M_0$. Hence $H_1(\N(G); \mathbb{Z}_2) \neq 0$.
If $p > 1$, then $C_{p-1} = \emptyset$ implies that  $\M_{p-1} = 0$. Thus, Ker($\partial_p$) $\cong \mathbb{Z}_2^{|C_{p}|}$, where 
$\partial_p : \M_p \longrightarrow \M_{p-1}$ is the boundary map. Since rank($\partial_{p+1}) < |C_{p}|$, we see that $H_p(\N(G); \mathbb{Z}_2) \neq 0$.
\end{proof}
We have developed all the necessary tools to prove the main result.
We recall the following results to prove Theorem \ref{main}.
\begin{prop} \label{prop7}(Theorem 3A.3, \cite{h}) 

 If $C$ is a chain complex of free abelian groups, then there exist short exact sequences
 \begin{center}
  $0 \longrightarrow H_n(C; \mathbb{Z}) \otimes \mathbb{Z}_2 \longrightarrow H_n(C; \mathbb{Z}_2) \longrightarrow$
  Tor$(H_{n-1}(C; \mathbb{Z}), \mathbb{Z}_2) \longrightarrow 0$
 \end{center}
 for all n and these sequences split.

\end{prop}
\begin{prop} \label{prop8}(The Hurewicz Theorem)\\
 If a space $X$ is $(n - 1)$ connected, $n \geq 2$, then $\tilde{H_i} (X; \mathbb{Z}) = 0$ for $i < n$
and $\pi_n (X) \cong H_n (X; \mathbb{Z})$.

\end{prop}


\noindent {\bf  Proof of Theorem \ref{main}.}

If $n = 2$, then $ \text{Hom} (K_2 \times K_2, K_m) \simeq  \text{Hom}(K_2 \sqcup K_2, K_m)
 \simeq$ Hom $(K_2, K_m)$ $\times$ Hom $(K_2, K_m) \simeq S^{m-2} \times S^{m-2}$.
 Hence $\text{conn}(\text{Hom}(K_2 \times K_2, $ $ K_m)) = \text{conn} (S^{m-2} \times S^{m-2}) = m-3$.

  Let $n \geq 3$. If $m < n$, \ref{lemma2} shows that $K_m^{K_n}$  can be folded to the graph $G'$, where
$V(G') = \{<x>  | \ x \in [m]\}$. Then $N(<x>) = \{<y> | \ y \in [m] \setminus \{x\} \}$, for all $<x> \ \in V(G')$
and therefore $\mathcal{N}(G')$ is homotopic to the simplicial boundary of $(m-1)$-simplex.
Hence $\text{Hom}(K_2 \times K_n, K_m) \simeq \N (K_m^{K_n}) \simeq \mathcal{N}(G') \simeq S^{m-2}$.
Therefore $\text{conn}(\text{Hom}(K_2 \times K_n, K_m)) = m-3$.

If $m = n$, then for any $f \in V(K_n^{K_n})$
with $\text{Im}\,f = [n]$, $N(f) = \{f\}$. Since $n \geq 2$, $\mathcal{N}(K_n^{K_n})$ is disconnected.

Assume $m-n = p \geq 1.$ 
Since $\text{conn} (\text{Hom}(K_2 \times K_n, K_m)) \geq 
m-n-1$, $\tilde{H_i} (\text{Hom}(K_2 \times K_n, K_m); \mathbb{Z}) = 0 \ \forall \ 0 \leq i \leq p-1.$ 
Since $\text{Hom}(K_2 \times K_n, K_m) \simeq \N(G), \tilde{H_i} (\N(G); \mathbb{Z}) = 0 \ \forall \
0 \leq i \leq p-1$ and $H_p(\text{Hom}(K_2 \times K_n, K_m); \mathbb{Z}_2) \neq 0$, by Theorem  \ref{main2}. 
By Proposition \ref{prop7}, $H_p (\N(G); \mathbb{Z}_2) \cong H_p(\N(G);\mathbb{Z}) \otimes \mathbb{Z}_2 \oplus
\text{Tor} (H_{p-1}(\N(G); \mathbb{Z}), \mathbb{Z}_2)$. Since $p \geq 1, \N(G)$ is path connected and hence 
$\text{Tor} (H_0(\N(G); \mathbb{Z}), \mathbb{Z}_2) \cong \text{Tor}(\mathbb{Z}, \mathbb{Z}_2) = 0$. Further, since 
for any $0 < q < p$, $\tilde{H_i}(\N(G); \mathbb{Z}) = 0$, $H_p(\N(G); \mathbb{Z}) = 0$ implies 
$H_p(\N(G); \mathbb{Z}_2) = 0$, which is a contradiction. Hence $H_p(\text{Hom}(K_2 \times K_n, K_m); \mathbb{Z}) 
\cong H_{p}(\N(G); \mathbb{Z}) \neq 0$.
If $p=1$, then since the abelinazitation of 
$\pi_1(\text{Hom}(K_2 \times K_n, K_m))$ is  $H_1(\text{Hom}$ $ (K_2 \times K_n, K_m); \mathbb{Z}) \neq 0$, 
$\pi_1(\text{Hom}(K_2 \times K_n, K_m)) \neq 0$.
For $p > 1$, since $\text{Hom} (K_2  \times K_n , K_m)$ is simply connected and 
$H_p(\text{Hom}(K_2 $ $ \times K_n, K_m); \mathbb{Z}) \neq 0$, the result follows from  Proposition \ref{prop8}.

\bibliographystyle{plain}

\end{document}